\documentclass[11pt,fleqn]{article}

\usepackage[dvipdfm]{graphicx}
\usepackage{epsfig}
\usepackage{amsmath}
\usepackage{amssymb}
\usepackage{amscd}
\usepackage{latexsym}
\usepackage{tabularx}
\usepackage{a4wide}
\usepackage{enumerate}
\usepackage{subfigure}
\usepackage{url}

\small\normalsize

\newtheorem{theorem}{Theorem}[section]
\newtheorem{corollary}[theorem]{Corollary}
\newtheorem{lemma}[theorem]{Lemma}
\newtheorem{conjecture}[theorem]{Conjecture}
\newtheorem{claim}[theorem]{Claim}
\newtheorem{proposition}[theorem]{Proposition}

\newcommand{\ch}{\mathit{ch}}
\newcommand{\ds}{d^\Sigma}
\newcommand{\eop}{\qquad\hspace*{\fill}\rule{1.2ex}{1.2ex}}
\newcommand{\eps}{\varepsilon}
\newcommand{\nog}{\widetilde{\mathbf{g}}}
\newcommand{\og}{\mathbf{g}}
\newcommand{\half}{\tfrac12}
\newcommand{\NOS}{\mathbb{N}}
\newcommand{\OS}{\mathbb{S}}
\newcommand{\MF}{\mathcal{F}}
\newcommand{\MM}{\mathcal{M}}
\newcommand{\MP}{\mathcal{MP}}

\newenvironment{proof}{\begin{trivlist}\item[]{\bf Proof}\mbox{ \ }}%
        {\qquad\hspace*{\fill}$\Box$\end{trivlist}}
\newcommand{\qite}[1]{\noindent\leavevmode\hangindent1\parindent%
        \noindent\hbox to1\parindent{#1\hss}\ignorespaces}
\newcommand{\qitee}[1]{\noindent\leavevmode\hangindent1.5\parindent%
        \noindent\hbox to1.5\parindent{#1\hss}\ignorespaces}
\newcommand{\qiteee}[1]{\noindent\leavevmode\hangindent2\parindent%
        \noindent\hbox to2\parindent{#1\hss}\ignorespaces}
\newcommand{\qittee}[1]{\noindent\leavevmode\hangindent3\parindent%
        \noindent\hbox to3\parindent{\hskip1.5\parindent#1\hss}%
        \ignorespaces}

\title{\textbf{A Unified Approach to Distance-Two Colouring\\[1mm]
    of Graphs on Surfaces}}

\author{\quad\\Omid Amini\,$^\ast$, \ Louis Esperet\,$^\dagger$, \ and \
  Jan van den Heuvel\,$^\ddagger$\\[3mm]
  $^\ast$\,\emph{CNRS -- DMA, \'Ecole Normale Sup\'erieure, Paris,
    France}\\[1mm]
  $^\dagger$\,\emph{CNRS -- Laboratoire G-SCOP, Grenoble,
    France}\\[1mm]
  $^\ddagger$\,\emph{Department of Mathematics, London School of Economics,
    London, U.K.}}

\date{ }

\begin{document}

\maketitle

{\renewcommand{\thefootnote}{\relax}

  \footnotetext{This paper benefited greatly from helpful comments of
    anonymous referees. The authors would like to thank the referees for
    careful reading of the paper and for their constructive suggestions.
    
    \hskip7.5pt The research for this paper was started during a visit of
    LE and JvdH to the Mascotte research group at INRIA Sophia-Antipolis,
    where OA was a PhD student (joint with \'Ecole Polytechnique). The
    authors like to thank the members of Mascotte for their hospitality.

    \hskip7.5pt JvdH's visit to INRIA Sophia-Antipolis was partly supported
    by a grant from the Alliance Programme of the British Council.

    \hskip7.5pt Part of this research has been conducted while OA was
    visiting McGill University in Montreal. He warmly thanks Bruce Reed for
    providing the possibility for such a visit.

    \hskip7.5pt Email: \texttt{oamini@math.ens.fr},
    \texttt{louis.esperet@g-scop.fr}, \texttt{jan@maths.lse.ac.uk}}}

\begin{abstract}
\noindent
In this paper we introduce the notion of $\Sigma$-colouring of a graph~$G$:
For given subsets~$\Sigma(v)$ of neighbours of~$v$, for every $v\in V(G)$,
this is a proper colouring of the vertices of~$G$ such that, in addition,
vertices that appear together in some~$\Sigma(v)$ receive different
colours. This concept generalises the notion of colouring the square of
graphs and of cyclic colouring of graphs embedded in a surface. We prove a
general result for graphs embeddable in a fixed surface, which implies
asymptotic versions of Wegner's and Borodin's Conjecture on the planar
version of these two colourings. Using a recent approach of Havet \emph{et
  al.}, we reduce the problem to edge-colouring of multigraphs, and then
use Kahn's result that the list chromatic index is close to the fractional
chromatic index.

Our results are based on a strong structural lemma for graphs embeddable in
a fixed surface, which also implies that the size of a clique in the square
of a graph of maximum degree~$\Delta$ embeddable in some fixed surface is
at most $\frac32\,\Delta$ plus a constant.
\end{abstract}

\clearpage
\section{Introduction}

Most of the terminology and notation we use in this paper is standard and
can be found in any text book on graph theory (such as~\cite{BoMu08}
or~\cite{Die05}). All our graphs and multigraphs will be finite. A
\emph{multigraph} can have multiple edges; a \emph{graph} is supposed to be
simple. We will not allow loops. The vertex and edge set of a graph~$G$ are
denoted by~$V(G)$ and~$E(G)$, respectively (or just~$V$ and~$E$, if the
graph~$G$ is clear from the context).

Given a graph~$G$, the \emph{chromatic number} of~$G$, denoted~$\chi(G)$,
is the minimum number of colours required so that we can properly colour
its vertices using those colours. If we colour the edges of~$G$, we get the
\emph{chromatic index}, denoted~$\chi'(G)$. The \emph{list chromatic
  number} or \emph{choice number}~$\mathit{ch}(G)$ is the minimum value~$k$
such that if we give each vertex~$v$ of~$G$ a list~$L(v)$ of at least~$k$
colours, then we can find a proper colouring in which each vertex gets
assigned a colour from its own private list. The list chromatic index
$\mathit{ch}'(G)$ is defined analogously for edges.

The \emph{square~$G^2$} of a graph~$G$ is the graph with vertex set~$V(G)$,
with an edge between any two different vertices that have distance at most
two in~$G$. A proper vertex colouring of the square of a graph can also be
seen as a vertex colouring of the original graph satisfying:

\qite{$\bullet$}vertices that are adjacent receive different colours, and

\qite{$\bullet$}vertices that have a common neighbour receive different
colours.

\noindent
Another way to formulate these conditions is as `vertices at distance one
or two must receive different colours'. This is why the name
\emph{distance-two colouring} is also used in the literature.

\smallskip
In this paper we consider a colouring concept that generalises the concept
of colouring the square of a graph, but that also can be used to study
different concepts such as \emph{cyclic colouring of plane graphs}
(definition will be given later).

For a vertex $v\in V$, let~$N(v)$ (or $N_G(v)$ if we want to specify the
graph under consideration) be the set of vertices adjacent to~$v$. Suppose
that for each vertex $v\in V$, we are given a subset
$\Sigma(v)\subseteq N(v)$ of its neighbourhood. We call such a collection a
\emph{$\Sigma$-system} for~$G$.

A \emph{$\Sigma$-colouring} of~$G$ is an assignment of colours to the
vertices of~$G$ so that:

\qite{$\bullet$}vertices that are adjacent receive different colours, and

\qite{$\bullet$}vertices that appear together in some~$\Sigma(v)$ receive
different colours.

\smallskip
When additionally each vertex~$v$ has its own list~$L(v)$ of colours from
which its colour must be chosen, we talk about a \emph{list
  $\Sigma$-colouring}.

We denote by $\chi(G;\Sigma)$ the minimum number of colours required for a
$\Sigma$-colouring to exist. Its list variant is denoted by
$\ch(G;\Sigma)$, and is defined as the minimum integer~$k$ such that for
each assignment of a list~$L(v)$ of at least~$k$ colours to vertices
$v\in V$, there exists a proper $\Sigma$-colouring of~$G$ in which all
vertices are assigned colours from their own lists.

Notice that we trivially have $\chi(G)=\chi(G;\varnothing)$ and
$\chi(G^2)=\chi(G;N_G)$; and the same relations holds for the list variant
($\varnothing$ assigns the empty set to each vertex).

We define the \emph{width} of a $\Sigma$-system of $G$ as
$\Delta(G;\Sigma)=\max_{v\in V}|\Sigma(v)|$. It is clear that we always
need at least $\Delta(G;\Sigma)+1$ colours in a proper $\Sigma$-colouring.
In the case $\Sigma\equiv N_G$, there exist plenty of graphs~$G$ that
require $O(\Delta(G)^2)$ colours (where $\Delta(G)=\Delta(G;N_G)$ is the
usual maximum degree of~$G$). But for planar graphs, it is known that a
constant times~$\Delta(G)$ colours is enough (even for list colouring). We
will take a closer look at this in Subsection~\ref{col-sq} below.

Following Wegner's Conjecture on colouring the square of planar graphs (see
also next subsection), we propose the following conjecture.

\begin{conjecture}\label{con1}\mbox{}\\*
  There exist constants $c_1,c_2$ and~$c_3$ such that for all planar
  graphs~$G$ and any $\Sigma$-system for~$G$, we have
  \begin{align*}
    \chi(G;\Sigma)\:&\le\:
    \bigl\lfloor\tfrac32\,\Delta(G;\Sigma)\bigr\rfloor+c_1;\\
    \ch(G;\Sigma)\:&\le\:
    \bigl\lfloor\tfrac32\,\Delta(G;\Sigma)\bigr\rfloor+c_2;\\
    \ch(G;\Sigma)\:&\le\:
    \bigl\lfloor\tfrac32\,\Delta(G;\Sigma)\bigr\rfloor+1,
    \qquad\qquad\text{if $\Delta(G;\Sigma)\ge c_3$}.
  \end{align*}
\end{conjecture}

\noindent
If $\Sigma\equiv\varnothing$ (hence $\Delta(G;\Sigma)=0$), then the Four
Colour Theorem implies that the smallest possible value for~$c_1$ is four;
while the fact that planar graphs are always 5-list colourable but not
always 4-list colourable, shows that the smallest possible value for~$c_2$
is five.

Our main result is that Conjecture~\ref{con1} is asymptotically correct:
$\ch(G;\Sigma)\le\frac32\,\Delta(G;\Sigma)+o\bigl(\Delta(G;\Sigma)\bigr)$.
In fact, we can prove this asymptotic result holds for general surfaces.

\begin{theorem}\label{th1}\mbox{}\\*
  For every surface~$S$ and any real $\eps>0$, there exists a
  constant~$\beta_{S,\eps}$ such that the following holds for all
  $\beta\ge\beta_{S,\eps}$. If~$G$ is a graph embeddable in~$S$, with a
  $\Sigma$-system of width at most $\beta$, then
  $\ch(G;\Sigma)\le\bigl(\frac32+\eps \bigr)\,\beta$.
\end{theorem}

\noindent
A trivial lower bound for the (list) chromatic number of a graph~$G$ is the
\emph{clique number~$\omega(G)$}, the maximum size of a clique in~$G$. For
graphs with a $\Sigma$-system, we can define the following related concept.
A \emph{$\Sigma$-clique} is a subset $C\subseteq V$ such that every two
different vertices in~$C$ are adjacent or appear together in
some~$\Sigma(v)$. Denote by $\omega(G;\Sigma)$ the maximum size of a
$\Sigma$-clique in~$G$. Then we trivially have
$\ch(G;\Sigma)\ge\omega(G;\Sigma)$, and so Theorem~\ref{th1} means that for
a graph~$G$ embeddable in some fixed surface~$S$, we have
$\omega(G;\Sigma)\le\frac32\,\Delta(G;\Sigma)+o(\Delta(G;\Sigma))$.

But in fact, the structural result we use to prove Theorem~\ref{th1} fairly
easily gives $\omega(G;\Sigma)\le\frac32\,\Delta(G;\Sigma)+O(1)$.

\begin{theorem}\label{th2}\mbox{}\\*
  For every surface~$S$, there exist constants~$\beta_S$ and~$\gamma_S$
  such that the following holds for all $\beta\ge\beta_S$. If~$G$ is a
  graph embeddable in~$S$, with a $\Sigma$-system of width at most $\beta$,
  then every $\Sigma$-clique in~$G$ has size at most
  $\frac32\,\beta+\gamma_S$.
\end{theorem}

\noindent
The main steps in the proof of Theorem~\ref{th1} can be found in
Section~\ref{proof}. The proof relies on two technical lemmas; the proofs
of those can be found in Section~\ref{thelemmas}. After that we use one of
those lemmas to provide the relatively short proof of Theorem~\ref{th2} in
Section~\ref{proof-cl}. In Section~\ref{conclusion} we discuss some of the
aspects of our work and discuss open problems related to (list)
$\Sigma$-colouring of graphs. The final section provides some background
regarding the proof by Kahn~\cite{Kah00} of the asymptotical equality of
the fractional chromatic index and the list chromatic index of multigraphs.
A more general result, contained implicitly in Kahn's work, is of crucial
importance to our proof in this paper.

In the next two subsections, we discuss two special consequences of these
results. These special versions of Theorems~\ref{th1} and~\ref{th2} also
show that the term $\frac32\,\beta$ is best possible.

\medskip
But before presenting these applications, a remark is in order. In an
earlier version of this paper, we gave our results in terms of
$(A,B)$-colourings. For a graph~$G$ and vertex sets $A,B\subseteq V$ (not
necessarily disjoint), an \emph{$(A,B)$-colouring} of~$G$ is a colouring of
the vertices in~$B$ such that adjacent vertices, and vertices with a common
neighbour in~$A$, receive different colours.

There is an obvious way to translate an $(A,B)$-colouring problem into a
$\Sigma$-colouring problem: For $v\in A$ set $\Sigma(v)=N_G(v)\cap B$, and
for $v\notin A$ set $\Sigma(v)=\varnothing$. Note that after this
translation we are required to colour all vertices, not just those in~$B$.
But the vertices outside~$B$ do not appear in any~$\Sigma(v)$, hence
colouring them for a graph embeddable in a fixed surface requires at most a
constant number of colours.

On the other hand, it is easy to construct instances of $\Sigma$-colouring
problems for which there is no obvious translation to an $(A,B)$-colouring
problem. In that sense, we feel justified in considering $\Sigma$-colouring
as a more general concept.

\subsection{Colouring the Square of Graphs}\label{col-sq}

Recall that the \emph{square} of a graph~$G$, denoted~$G^2$, is the graph
with the same vertex set as~$G$ and with an edge between any two different
vertices that have distance at most two in~$G$. If~$G$ has maximum
degree~$\Delta$, then a vertex colouring of its square will need at least
$\Delta+1$ colours, and the greedy algorithm shows that it is always
possible to find a colouring of~$G^2$ with $\Delta^2+1$ colours. Cages of
diameter two, such as the 5-cycle, the Petersen graph and the
Hoffman-Singleton graph (see, e.g., \cite[page~84]{BoMu08}), show that
there exist graphs that in fact require $\Delta^2+1$ colours.

Regarding the chromatic number of the square of a planar graph,
Wegner~\cite{Weg77} posed the following conjecture (see also the book of
Jensen and Toft~\cite[Section 2.18]{JeTo95}), suggesting that for planar
graphs far less than $\Delta^2+1$ colours suffice.

\begin{conjecture}[Wegner~\cite{Weg77}]\label{con-weg}\mbox{}\\*
  For a planar graph~$G$ of maximum degree~$\Delta$,
  $\chi(G^2)\:\le\:\left\{
    \begin{array}{ll}
      7,&\text{if $\Delta =3$,}\\
      \Delta+5,&\text{if $4\le\Delta\le7$,}\\
      \bigl\lfloor\frac32\,\Delta\bigr\rfloor+1,&\text{if $\Delta\ge8$.}
    \end{array}\right.$
\end{conjecture}

\noindent
Wegner also gave examples showing that these bounds would be tight. For
$\Delta\ge8$ even, these examples are sketched in Figure~\ref{fig:exwegner}.
\begin{figure}[htbp]
\begin{center}
\hspace{0.05cm}
\subfigure[\label{fig:exwegner}]{\unitlength0.45mm
    \begin{picture}(82,90)(-50,-41)
      \put(-50,0){\circle*{3}}\put(-44,0){\circle*{3}}
      \put(-38,0){\circle*{3}}\put(-33,0){\circle*{1.5}}
      \put(-29,0){\circle*{1.5}}\put(-25,0){\circle*{1.5}}
      \put(-21,0){\circle*{1.5}}\put(-16,0){\circle*{3}}
      \put(25,43.30){\circle*{3}}\put(22,38.11){\circle*{3}}
      \put(19,32.91){\circle*{3}}\put(16.5,28.58){\circle*{1.5}}
      \put(14.5,25.11){\circle*{1.5}}\put(12.5,21.65){\circle*{1.5}}
      \put(10.5,18.19){\circle*{1.5}}\put(8,13.86){\circle*{3}}
      \put(25,-43.30){\circle*{3}}\put(22,-38.11){\circle*{3}}
      \put(19,-32.91){\circle*{3}}\put(16.5,-28.58){\circle*{1.5}}
      \put(14.5,-25.11){\circle*{1.5}}\put(12.5,-21.65){\circle*{1.5}}
      \put(10.5,-18.19){\circle*{1.5}}\put(8,-13.86){\circle*{3}}
      \put(-16,27.71){\circle*{3.5}}\put(-16,-27.71){\circle*{3.5}}
      \put(32,0){\circle*{3.5}}
      \qbezier(-16,27.71)(-16,13.855)(-16,0)
      \qbezier(-16,27.71)(-27,13.855)(-38,0)
      \qbezier(-16,27.71)(-30,13.855)(-44,0)
      \qbezier(-16,27.71)(-33,13.855)(-50,0)
      \qbezier(-16,-27.71)(-16,-13.855)(-16,0)
      \qbezier(-16,-27.71)(-27,-13.855)(-38,0)
      \qbezier(-16,-27.71)(-30,-13.855)(-44,0)
      \qbezier(-16,-27.71)(-33,-13.855)(-50,0)
      \qbezier(-16,27.71)(-4,20.785)(8,13.86)
      \qbezier(-16,27.71)(1.5,29.95)(19,32.91)
      \qbezier(-16,27.71)(3,32.91)(22,38.11)
      \qbezier(-16,27.71)(4.5,35.505)(25,43.30)
      \qbezier(-16,-27.71)(-4,-20.785)(8,-13.86)
      \qbezier(-16,-27.71)(1.5,-29.95)(19,-32.91)
      \qbezier(-16,-27.71)(3,-32.91)(22,-38.11)
      \qbezier(-16,-27.71)(4.5,-35.505)(25,-43.30)
      \qbezier(32,0)(20,6.93)(8,13.86)
      \qbezier(32,0)(25.5,16.455)(19,32.91)
      \qbezier(32,0)(27,19.055)(22,38.11)
      \qbezier(32,0)(28.5,21.65)(25,43.30)
      \qbezier(32,0)(20,-6.93)(8,-13.86)
      \qbezier(32,0)(25.5,-16.455)(19,-32.91)
      \qbezier(32,0)(27,-19.055)(22,-38.11)
      \qbezier(32,0)(28.5,-21.65)(25,-43.30)
      \qbezier(-16,27.71)(0,0)(-16,-27.71)
      \put(-54,4){\makebox(0,0)[r]{$k-1$}}
      \put(-54,-4){\makebox(0,0)[r]{vertices}}
      \put(29,43.40){\makebox(0,0)[l]{$k$ vertices}}
      \put(29,-43.40){\makebox(0,0)[l]{$k$ vertices}}
      \put(38,0){\makebox(0,0){$z$}}
      \put(-19,32.91){\makebox(0,0){$x$}}
      \put(-19,-32.91){\makebox(0,0){$y$}}
    \end{picture}}
\hspace{4cm}
\subfigure[\label{fig:exboro}]{\unitlength0.45mm
    \begin{picture}(82,90)(-50,-41)
      \put(-50,0){\circle*{3}}\put(-44,0){\circle*{3}}
      \put(-38,0){\circle*{3}}\put(-33,0){\circle*{1.5}}
      \put(-29,0){\circle*{1.5}}\put(-25,0){\circle*{1.5}}
      \put(-21,0){\circle*{1.5}}\put(-16,0){\circle*{3}}
      \put(25,43.30){\circle*{3}}\put(22,38.11){\circle*{3}}
      \put(19,32.91){\circle*{3}}\put(16.5,28.58){\circle*{1.5}}
      \put(14.5,25.11){\circle*{1.5}}\put(12.5,21.65){\circle*{1.5}}
      \put(10.5,18.19){\circle*{1.5}}\put(8,13.86){\circle*{3}}
      \put(25,-43.30){\circle*{3}}\put(22,-38.11){\circle*{3}}
      \put(19,-32.91){\circle*{3}}\put(16.5,-28.58){\circle*{1.5}}
      \put(14.5,-25.11){\circle*{1.5}}\put(12.5,-21.65){\circle*{1.5}}
      \put(10.5,-18.19){\circle*{1.5}}\put(8,-13.86){\circle*{3}}
      \put(-50,0){\line(1,0){34}}
      \qbezier(25,43.30)(16.5,28.58)(8,13.86)
      \qbezier(25,-43.30)(16.5,-28.58)(8,-13.86)
      \put(25,-43.30){\line(0,1){86.6}}
      \put(8,-13.86){\line(0,1){27.72}}
      \qbezier(-50,0)(-12.5,21.65)(25,43.30)
      \qbezier(-50,0)(-12.5,-21.65)(25,-43.30)
      \qbezier(-16,0)(-4,6.93)(8,13.86)
      \qbezier(-16,0)(-4,-6.93)(8,-13.86)
      \put(-56,0){\makebox(0,0)[r]{$k$ vertices}}
      \put(29,43.40){\makebox(0,0)[l]{$k$ vertices}}
      \put(29,-43.40){\makebox(0,0)[l]{$k$ vertices}}
    \end{picture}}
  \caption{(a) A planar graph~$G$ with maximum degree $\Delta=2k$ and
    $\omega(G^2)=\chi(G^2)=3k+1=
    \bigl\lfloor\frac32\,\Delta\bigr\rfloor+1$.\newline
    (b)~A planar graph~$H$ with maximum face order $\Delta^*=2k$ and
    $\chi^*(H)=3k=\bigl\lfloor\frac32\,\Delta^*\bigr\rfloor$ (see
    Subsection~\ref{col-cycl}).
    \label{fig:exwegnerboro}}
\end{center}
\end{figure}
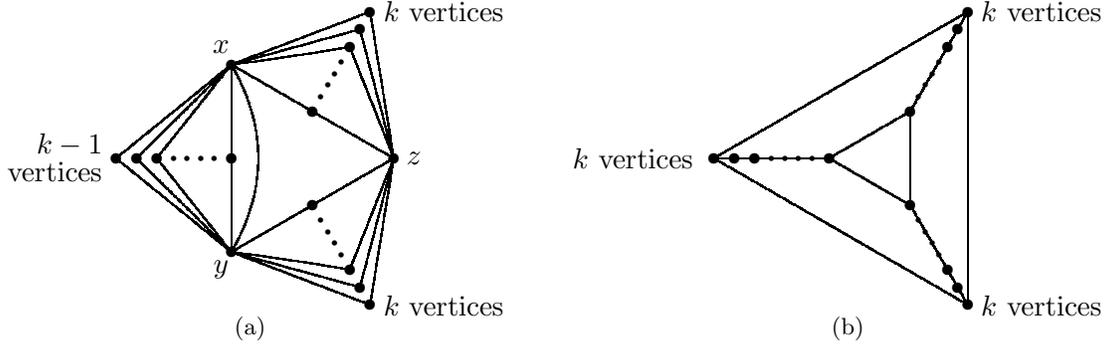
The graph in the picture has maximum degree~$2k$ and yet all the vertices
except~$z$ are pairwise adjacent in its square. Hence to colour these
$3k+1$ vertices, we need at least $3k+1=\frac32\,\Delta+1$ colours. Note
that the same arguments also show that the graph~$G$ in the picture has
$\omega(G^2)=\frac32\,\Delta+1$.

Kostochka and Woodall~\cite{KoWo01} conjectured that for the square of any
graph, the chromatic number equals the list chromatic number. This
conjecture and Wegner's one together imply the conjecture that for planar
graphs~$G$ with $\Delta\ge8$, we have
$\ch(G^2)\le\bigl\lfloor\frac32\,\Delta\bigr\rfloor+1$.

The first upper bound on~$\chi(G^2)$ for planar graphs in terms
of~$\Delta$, $\chi(G^2)\le8\Delta-22$, was implicit in the work of
Jonas~\cite{Jon93}. This bound was later improved by Wong~\cite{Won96} to
$\chi(G^2)\le3\Delta+5$ and then by Van den Heuvel and
McGuinness~\cite{vdHMG} to $\chi(G^2)\le2\Delta+25$. Better bounds were
then obtained for large values of~$\Delta$. It was shown by Agnarsson and
Halld\'orsson~\cite{AgHa00} that for $\Delta\ge750$ we have
$\chi(G^2)\le\lceil\frac95\,\Delta\rceil+1$; the same bound was proved for
$\Delta\ge47$ by Borodin \emph{et al.}~\cite{Bor+}. Finally, the best known
upper bound so far has been obtained by Molloy and
Salavatipour~\cite{MoSa02}: $\chi(G^2)\le\lceil\frac53\,\Delta\rceil+78$.
As mentioned in \cite{MoSa02}, the constant 78 can be reduced for
sufficiently large~$\Delta$. For example, it was improved to~24 when
$\Delta\ge241$.

\medskip
Since $\ch(G^2)=\ch(G; N_G)$ (i.e., $\Sigma(v)=N_G(v)$ for all $v\in V$),
as an immediate corollary of Theorem~\ref{th1} we obtain.

\begin{corollary}\label{cor1}\mbox{}\\*
  Let~$S$ be a fixed surface. Then the square of every graph~$G$ embeddable
  in~$S$ and of maximum degree~$\Delta$ has list chromatic number at most
  $\frac32\,\Delta+o(\Delta)$.
\end{corollary}

\noindent
In fact, the same asymptotic upper bound as in Corollary~\ref{cor1} can be
proved even for larger classes of graphs. Additionally, a stronger
conclusion on the colouring is possible. For the following result, we
assume that colours are integers, which allows us to talk about the
`distance' $|\alpha_1-\alpha_2|$ between two colours~$\alpha_1,\alpha_2$.

\begin{theorem}[Havet, Van den Heuvel, McDiarmid \& Reed
  \cite{HHMR}]\label{mt2}\mbox{}\\*
  Let~$k$ be a fixed positive integer. The square of every $K_{3,k}$-minor
  free graph~$G$ of maximum degree~$\Delta$ has list chromatic number (and
  hence clique number) at most $\frac32\,\Delta+o(\Delta)$. Moreover, given
  lists of this size, there is a proper colouring in which the colours on
  every pair of adjacent vertices of~$G$ differ by at least~$\Delta^{1/4}$.
\end{theorem}

\noindent
Note that planar graphs do not have a $K_{3,3}$-minor. In fact, for every
surface~$S$, there is a constant~$k$ such that no graph embeddable in~$S$
has $K_{3,k}$ as a minor. This shows that Theorem~\ref{mt2} is stronger
than our Corollary~\ref{cor1}. On the other hand, Theorem~\ref{mt2} gives a
weaker bound for the clique number than the one we obtain in
Corollary~\ref{cor1-cl} below.

Both Corollary~\ref{cor1} and Theorem~\ref{mt2} can be applied to
$K_4$-minor free graphs, since these graphs are planar and do not
have~$K_{3,3}$ as a minor. But the best possible bounds for this class are
actually known. Lih, Wang and Zhu~\cite{LWZ03} showed that the square of
$K_4$-minor free graphs with maximum degree~$\Delta$ has chromatic number
at most $\bigl\lfloor\frac32\,\Delta\bigr\rfloor+1$ if $\Delta\ge4$ and
$\Delta+3$ if $\Delta=2,3$. The same bounds, but then for the list
chromatic number of $K_4$-minor free graphs, were proved by Hetherington
and Woodall~\cite{HW07}.

\medskip
Regarding the clique number of the square of graphs, we get the following
corollary of Theorem~\ref{th2}.

\begin{corollary}\label{cor1-cl}\mbox{}\\*
  Let~$S$ be a fixed surface. Then the square of every graph~$G$ embeddable
  in~$S$ and of maximum degree~$\Delta$ has clique number at most
  $\frac32\,\Delta+O(1)$.
\end{corollary}

\noindent
{}From the proof of Theorem~\ref{th2}, it can be deduced that the square of
a planar graph with maximum degree $\Delta\ge11616$ has clique number at
most $\frac32\,\Delta+76$.

Very recently, this was improved by the following result.

\begin{theorem}[Cohen \& Van den Heuvel
  \cite{CvdH}]\label{new-cl}\mbox{}\\*
  For a planar graph~$G$ of maximum degree $\Delta\ge41$, we have
  $\omega(G^2)\le\bigl\lfloor\frac32\,\Delta\bigr\rfloor+1$.
\end{theorem}

\noindent
Apart from the bound $\Delta\ge41$, this theorem is best possible, as is
shown by the same graphs that show that Wegner's Conjecture~\ref{con-weg}
is best possible for $\Delta\ge8$ (see also Figure~\ref{fig:exwegner}).

\subsection{Cyclic Colourings of Embedded Graphs }\label{col-cycl}

Given a surface~$S$ and a graph~$G$ embeddable in~$S$, we denote by~$G^S$
that graph with a prescribed embedding in~$S$. If the surface~$S$ is the
sphere, we talk about a \emph{plane graph~$G^P$}. The \emph{order} of a
face of~$G^S$ is the number of vertices in its boundary; the maximum order
of a face of~$G^S$ is denoted by $\Delta^*(G^S)$.

A \emph{cyclic colouring} of an embedded graph~$G^S$ is a vertex colouring
of~$G$ such that any two vertices in the boundary of the same face have
distinct colours. The minimum number of colours required in a cyclic
colouring of an embedded graph is called the \emph{cyclic chromatic
  number~$\chi^*(G^S)$}. This concept was introduced for plane graphs by
Ore and Plummer~\cite{OP69}, who also proved that for a plane graph~$G^P$,
we have $\chi^*(G^P)\le2\Delta^*$. Borodin~\cite{Bor84} (see also Jensen
and Toft~\cite[page~37]{JeTo95}) conjectured the following.

\begin{conjecture}[Borodin~\cite{Bor84}]\label{conb}\mbox{}\\*
  For a plane graph~$G^P$ of maximum face order~$\Delta^*$, we have
  $\chi^*(G^P)\le\bigl\lfloor\frac32\,\Delta^*\bigr\rfloor$.
\end{conjecture}

\noindent
The bound in this conjecture is best possible. Consider the plane graph
depicted in Figure~\ref{fig:exboro}: It has $3k$ vertices and has three
faces of order $\Delta^*=2k$. Since all pairs of vertices have a face they
are both incident with, we need
$3k=\bigl\lfloor\frac32\,\Delta^*\bigr\rfloor$ colours in a cyclic
colouring.

Borodin~\cite{Bor84} also proved Conjecture~\ref{conb} for $\Delta^*=4$.
For general values of~$\Delta^*$, the original bound
$\chi^*(G^P)\le2\Delta^*$ of Ore and Plummer~\cite{OP69} was improved by
Borodin \emph{et al.}~\cite{BorSZ} to
$\chi^*(G^P)\le\bigl\lfloor\frac95\,\Delta^*\bigr\rfloor$. The best known
upper bound in the general case is due to Sanders and Zhao~\cite{San01}:
$\chi^*(G^P)\le\bigl\lceil\frac53\,\Delta^*\bigr\rceil$.

Although Wegner's and Borodin's Conjectures seem to be closely related,
nobody has ever been able to bring to light a direct connection between
them. Most of the results approaching these conjectures use the same ideas,
but up until this point no one had proved a general theorem implying both a
result on the colouring of the square and a result on the cyclic colouring
of plane graphs (let alone on embedded graphs).

In order to show that our Theorem~\ref{th1} provides an asymptotically best
possible upper bound for the cyclic chromatic number for a graph~$G$ with
some fixed embedding~$G^S$, we need some extra notation. For each face~$f$
of~$G^S$, add a vertex~$x_f$. For any face~$f$ of~$G^S$ and any vertex~$v$
in the boundary of~$f$, add an edge between~$v$ and~$x_f$, and denote
by~$G_F$ the graph obtained from~$G^S$ by this construction. Note that the
vertex set of~$G_F$ consists of~$V(G)$ and all the new vertices~$x_f$,
for~$f$ a face of~$G^S$. Define a $\Sigma$-system~$\Sigma_F$ for~$G_F$ as
follows: For each vertex $v\in V(G)$, let $\Sigma_F(v)=\varnothing$. For
each vertex~$x_f$, let $\Sigma_F(x_f)$ be all the neighbours of~$x_f$.
Observe that a (list) $\Sigma_F$-colouring of~$G_F$ colours the vertices
of~$G$ in a way required for a cyclic (list) colouring of~$G^S$, and that
$\Delta(G_F;\Sigma_F)=\Delta^*(G^S)$.

(Note that in fact we have
$\chi^*(G^S)\le\chi(G_F,\Sigma_F)\le\chi^*(G^S)+1$. To get the second
inequality, start with a cyclic colouring of~$G^S$, add one extra colour,
and colour all the vertices~$x_f$ with that colour. Similar inequalities
hold for the list version.)

Using the upper bound on $\chi^*(G^S)$, we get the following corollary of
Theorem~\ref{th1}.

\begin{corollary}\label{cor2}\mbox{}\\*
  Let~$S$ be a fixed surface. Every embedding~$G^S$ of a graph~$G$ of
  maximum face order~$\Delta^*$ has cyclic list chromatic number at most
  $\frac32\,\Delta^*+o(\Delta^*)$.
\end{corollary}

\noindent
For an embedded graph~$G^S$, the \emph{cyclic clique
  number~$\omega^*(G^S)$} is the maximum size of a set $C\subseteq V$ such
that every two vertices in~$C$ have some face they are both incident with.
Note that the plane graph depicted in Figure~\ref{fig:exboro} satisfies
$\omega^*(G^P)=3k=\bigl\lfloor\frac32\,\Delta^*\bigr\rfloor$. This shows
that the following corollary of Theorem~\ref{th2} is best possible, up to
the constant term.

\begin{corollary}\label{cor2-cl}\mbox{}\\*
  Let~$S$ be a fixed surface. Every embedded graph~$G^S$ of maximum face
  order~$\Delta^*$ has cyclic clique number at most
  $\frac32\,\Delta^*+O(1)$.
\end{corollary}

\noindent
For plane graphs, the proof of Theorem~\ref{th2} guarantees that a plane
graph~$G^P$ of maximum face order $\Delta^*\ge11616$ has cyclic clique
number at most $\frac32\,\Delta^*+76$.

\section{Proof of Theorem~\ref{th1}}\label{proof}

Our goal in this section is to show that for all surfaces~$S$ and any
$\eps>0$, if we take~$\beta$ large enough (depending on~$S$ and~$\eps$),
then for every graph $G=(V,E)$ embeddable in~$S$, every choice of
$\Sigma(v)\subseteq N_G(v)$ with $|\Sigma(v)|\le\beta$ for all $v\in V$,
and every assignment~$L(v)$ of at least $\bigl(\frac32+\eps\bigr)\,\beta$
colours to all $v\in V$, there is a list $\Sigma$-colouring of~$G$ where
each vertex receives a colour from its own list. In other words, we want an
assignment~$c(v)$ for each $v\in V$ such that:

\qite{$\bullet$}for all $v\in V$, we have $c(v)\in L(v)$;

\qite{$\bullet$}for all $u,v\in V$ with $uv\in E$, we have $c(u)\ne c(v)$;
and

\qite{$\bullet$}for all $u,v\in V$ for which there is a $t\in V$ with
$u,v\in\Sigma(t)$, we have $c(u)\ne c(v)$.

\medskip
Before we present the actual proofs, we recall some of the important
terminology, notation and facts concerning embeddings of graph in surfaces.

\subsection{Graphs in Surfaces}\label{gr-on-sur}

In this subsection, we give some background about graphs embedded in a
surface. For more details, the reader is referred to~\cite{MoTh}. Here, by
a \emph{surface} we mean a compact 2-dimensional surface without boundary.
An \emph{embedding} of a graph~$G$ in a surface~$S$ is a drawing of~$G$
on~$S$ so that all vertices are distinct, and every edge forms a simple arc
connecting in~$S$ the vertices it joins, so that the interior of every edge
is disjoint from other vertices and edges. A \emph{face} of this embedding
(or just a \emph{face of~$G$}, for short) is an arc-wise connected
component of the space obtained by removing the vertices and edges of~$G$
from the surface~$S$.

We say that an embedding is \emph{cellular} if every face is homeomorphic
to an open disc in~$\mathbb{R}^2$.

A surface can be orientable or non-orientable. The \emph{orientable
  surface~$\OS_h$ of genus~$h$} is obtained by adding $h\ge0$ `handles' to
the sphere; while the \emph{non-orientable surface~$\NOS_k$ of genus~$k$}
is formed by adding $k\ge1$ `cross-caps' to the sphere. The
\emph{genus~$\og(G)$} and \emph{non-orientable genus~$\nog(G)$} of a
graph~$G$ is the minimum~$h$ and the minimum~$k$, resp., such that~$G$ has
an embedding in~$\OS_h$, resp.\ in~$\NOS_k$.

The following result will allow us to suppose that a graph~$G$ with known
genus~$\og(G)$ or non-orientable genus~$\nog(G)$ can be assumed to be
embedded in a cellular way.

\begin{lemma}[{\cite[Propositions~3.4.1
    and~3.4.2]{MoTh}}]\label{lem-embedding}\mbox{}\\*
  \hspace*{-1.5\parindent}%
  \qitee{(i)}Every embedding of a connected graph~$G$ in~$\OS_{\og(G)}$ is
  cellular.

  \smallskip
  \qitee{(ii)}If~$G$ is a connected graph different from a tree, then there
  is an embedding of~$G$ in~$\NOS_{\nog(G)}$ that is cellular.
\end{lemma}

\noindent
The \emph{Euler characteristic~$\chi(S)$} of a surface~$S$ is $2-2h$ if
$S=\OS_h$, and $2-k$ if $S=\NOS_k$.

The basic result connecting all these concepts is \emph{Euler's Formula}:
If~$G$ is a graph with an embedding in~$S$, with vertex set~$V$, edge
set~$E$ and face set~$F$, then
\[|V|-|E|+|F|\:\ge\:\chi(S).\]
Moreover, if the embedding is cellular, then we have equality in Euler's
Formula.

Finally, if~$v$ is a vertex of a graph~$G$ embedded in a surface~$S$, then
that embedding imposes two circular orders of the edges incident with~$v$.
Since we assume graphs to be simple, this corresponds to two circular
orders of the neighbours of~$v$. If~$S$ is orientable, then we can
consistently choose one of the two clockwise orders for \emph{all}
vertices; if~$S$ is non-orientable, then such a choice is not possible. In
our proofs that follow, it is not important that we can choose a consistent
circular order; we only require that for each vertex~$v$, there is at least
one circular order of the neighbours around~$v$.

If $u_1,u_2$ are consecutive neighbours of~$v$ (with respect to the chosen
circular order), then there is a face that has the three vertices
$u_1,v,u_2$ in its boundary. That immediately gives the following
observation.

\begin{lemma}\label{lem0}\mbox{}\\*
  Let~$G$ be a graph embedded in a surface~$S$. Suppose $u_1,u_2$ are
  consecutive neighbours of~$v$ (with respect to the chosen circular
  order). Then the graph obtained by adding the edge~$u_1u_2$ (if it is not
  already present) is still embeddable in~$S$.
\end{lemma}

\noindent
That observation has the following corollary.

\begin{lemma}\label{lem-1}\mbox{}\\*
  Let~$G$ be a connected graph embedded in a surface~$S$. If~$G$ has more
  than three vertices and is edge-maximal with respect to being embeddable
  in~$S$, then every vertex has degree at least three.
\end{lemma}

\subsection{The First Steps}\label{sec2.1}

For $P,Q\subseteq V$, the set of edges between~$P$ and~$Q$ is denoted by
$E(P,Q)$, and the number of edges between~$P$ and~$Q$ is denoted by
$e(P,Q)$ (edges with both ends in $P\cap Q$ are counted twice).

For a graph~$G$ with a $\Sigma$-system, and a vertex $v\in V$, we denote
by~$\sigma(v)$ the size of~$\Sigma(v)$, i.e., $\sigma(v)=|\Sigma(v)|$. A
\emph{$\Sigma$-neighbour} of a vertex~$v$ is a vertex $u\ne v$ such that
either~$u$ and~$v$ are adjacent, or there is some $t\in V$ with
$u,v\in\Sigma(t)$. Denote the number of $\Sigma$-neighbours of~$v$
by~$\ds(v)$. Note that we have
\[\ds(v)\:\le\:d(v)+\!\!
\sum_{t\text{ with }v\,\in\,\Sigma(t)}\!\!(\sigma(t)-1).\]

An important tool in our proof of Theorem~\ref{th1} is the following
technical structural result, Lemma~\ref{lem1}. Before stating this lemma,
we need a few extra definitions. For an integer~$\zeta$, a \emph{special
  $\zeta$-pair} is a pair $(X,Y)$ of disjoint subsets of vertices~$X$
and~$Y$ (possibly empty) with the following property:

\smallskip
\qittee{(i)}\emph{Every vertex in~$X$ has degree at least $\zeta+1$. Every
  vertex $y\in Y$ has degree four, is adjacent to exactly two vertices
  of~$X$, and the remaining neighbours of~$y$ have degree four as well.}

\smallskip
Given a special $\zeta$-pair $(X,Y)$, for any $y\in Y$, let~$X^y$ be the
set of two neighbours of $y$ in~$X$. For $W\subseteq X$, let~$Y^W$ be the
set of all vertices $y\in Y$ with $X^y\subseteq W$ (that is, the set of
vertices of~$Y$ having their two neighbours from~$X$ in~$W$).
  
A special $\zeta$-pair $(X,Y)$ is called \emph{very special} if in addition
the following condition holds:

\smallskip
\qittee{(ii)}\emph{For all pairs of vertices $y,z\in Y$, if~$y$ and~$z$ are
  adjacent or have a common neighbour $w\notin X$, then $X^y=X^z$.}

\smallskip
The general structure of a very special $\zeta$-pair is sketched in
Figure~\ref{fig:struct}.
\begin{figure}[htbp]
  \begin{center}
    \includegraphics[scale=0.8]{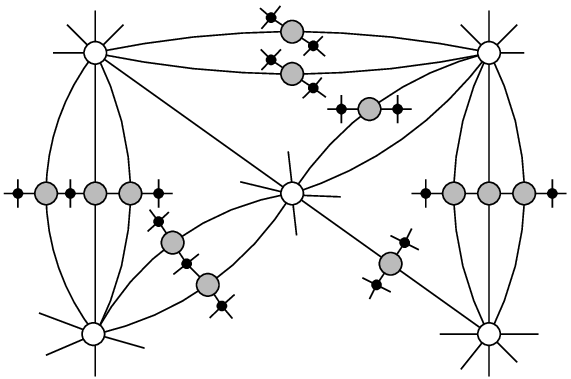}
    \caption{Sets $X$ (white vertices) and $Y$ (grey vertices) forming a
      very special 8-pair $(X,Y)$ in a graph. Neighbours of vertices of~$Y$
      not in~$X$ are depicted as small black vertices, and the remaining
      vertices are not depicted for the sake of clarity.\label{fig:struct}}
  \end{center}
\end{figure}

\medskip 
With these definitions, our structural lemma can be stated as follows:
 
\begin{lemma}\label{lem1}\mbox{}\\*
  Let~$S$ be a fixed surface, set $\zeta_S^*=132\,(3-\chi(S))$, and let~$G$
  be a graph embeddable in~$S$. If~$G$ is edge-maximal with respect to
  being embeddable in~$S$, then one of the following three properties
  holds.
    
  \smallskip
  \qitee{(S1)}Every vertex has degree at most~$\zeta_S^*$.

  \smallskip
  \qitee{(S2)}There is a vertex of degree at most five with at most one
  neighbour of degree more than~$\zeta_S^*$.

  \smallskip
  \qitee{(S3)}There exists a very special $\zeta_S^*$-pair $(X,Y)$ such
  that $X,Y$ are both non-empty and for all non-empty subsets
  $W\subseteq X$, the following inequality holds:
  \[e(W,V\setminus Y)\:\le\:e(W,Y\setminus Y^W)+\zeta_S^*\,|W|.\]
\end{lemma}

\noindent
Very informally, Lemma~\ref{lem1} states that a graph that is maximally
embeddable in some fixed surface, either contains one of two fairly simple
configurations, or it contains a structure that internally satisfies a
specific density-type condition.

Structure~(S3) is at the heart of the above lemma. Although its description
might appear technical at first sight, it will be clear later that it is
the exact kind of density condition needed in the proofs of
Theorems~\ref{th1} and~\ref{th2}.
  
The proof of Lemma~\ref{lem1} can be found in Subsection~\ref{proof-lem1}.
Observe that the value we use for~$\zeta_S^*$ is probably far from best
possible. The important point, to our mind, is that it only depends on (the
Euler characteristic of) the surface~$S$.

\medskip
We continue with a description how to apply the lemma to prove
Theorem~\ref{th1}. Suppose the theorem is false. Then there is a
surface~$S$ and a real $\eps>0$ such that for every~$\beta_{S,\eps}$ we can
find $\beta\ge\beta_{S,\eps}$ and a graph~$G$, together with a
$\Sigma$-system of width at most $\beta$, such that
$\mathit{ch}(G;\Sigma)>\bigl(\frac32+\eps\bigr)\,\beta$. Set
$\zeta_S^*=132\,(3-\chi(S))$ and
$\beta_S^*=\tfrac23\,(\zeta_S^*)^2=11616\,(3-\chi(S))^2$. Note that, as
$\chi(S)\le2$, this means $\zeta_S^*\ge132$ and $\beta_S^*\ge11616$.

We start by assuming $\beta\ge\beta_S^*$; later (at the end of
Subsection~\ref{sec-mp}) we will add some further lower bounds for~$\beta$
that will depend on~$\eps$. With respect to this (yet to come) final choice
of~$\beta$, there exists a graph $G=(V,E)$ embeddable in $S$, together with
a $\Sigma$-system of width at most $\beta$ and a list-assignment~$L$ of at
least $\bigl(\frac32+\eps\bigr)\,\beta$ colours to each vertex $v\in V$,
such that $G$ has no $\Sigma$-colouring from these lists. Choose such a
graph~$G$ with the minimum number of vertices, and subject to this, with
the maximum number of edges.

Certainly we can assume that~$G$ is connected (otherwise one of the
components will be a smaller counterexample). Also, since each vertex has a
list of more than $\frac32\,\beta\ge17424$ colours, $G$ itself will have
more than 17424 vertices.

Next we can assume that~$G$ is edge-maximal with respect to being
embeddable in~$S$. Otherwise we can add a new edge~$uv$ to~$G$ so that the
resulting graph~$G_1$ is still embeddable in~$S$, and set
$\Sigma_1\equiv\Sigma$. It is clear that a list $\Sigma_1$-colouring
of~$G_1$ is also is a list $\Sigma$-colouring of~$G$.

Fix some embedding of~$G$ in~$S$. We continue with applying
Lemma~\ref{lem1} to~$G$.

\subsubsection{The structure from (S1) is present
  in~{\boldmath$G$}}\label{ss-S1}

This is the easiest case: If the degree of every vertex is at
most~$\zeta_S^*$, then the number of $\Sigma$-neighbours of any vertex is
at most $\zeta_S^*+\zeta_S^*\cdot(\zeta_S^*-1)=(\zeta_S^*)^2$. But the
number of colours in each list~$L(v)$ is at least
$\bigl(\frac32+\eps\bigr)\,\beta>\frac32\,\beta_S^*=(\zeta_S^*)^2$. So a
simple greedy colouring will do the job; contradicting that~$G$ is a
counterexample.

\subsubsection{The structure from (S2) is present
  in~{\boldmath$G$}}\label{ss-S2}

So there is a vertex~$v$ of degree at most five, and at most one of its
neighbours has degree more than~$\zeta_S^*$. Since $|V|\ge17424$, by
Lemma~\ref{lem-1}, $v$ has degree at least three. Hence it has a
neighbour~$u$ of degree at most~$\zeta_S^*$. Form the graph~$G_2$ by
contracting~$uv$ into a new vertex~$w$ (remove multiple edges if they
appear). Set $V_2=(V\setminus\{u,v\})\cup\{w\}$. Let
$\Sigma_2(w)=(\Sigma(u)\cup\Sigma(v))\setminus\{u,v\}$. For a vertex
$t\in V_2\setminus\{w\}$, if $\Sigma(t)$ contains~$u$, then set
$\Sigma_2(t)=(\Sigma(t)\setminus\{u,v\})\cup\{w\}$; otherwise, set
$\Sigma_2(t)=\Sigma(t)\setminus\{v\}$. Finally, give~$w$ the list of
colours $L(w)=L(u)$. Note that~$G_2$ is smaller than~$G$ and is still
embeddable in~$S$. Moreover, for every $t\in V_2\setminus\{w\}$, we have
$|\Sigma_2(t)|\le|\Sigma(t)|\le\beta$; while for~$w$ we have
$|\Sigma_2(w)|\le d_G(u)+d_G(v)\le5+\zeta_S^*\le\beta$.

So there exists a list $\Sigma_2$-colouring of~$G_2$. We define a colouring
of~$G$ as follows: Every vertex different from~$u$ and~$v$ keeps its colour
from the colouring of~$G_2$. We give~$u$ the colour given to~$w$ in~$G_2$.
Finally, we observe that for~$v$ we have
\[\ds(v)\:\le\:d(v)+\!\!\!
\sum_{t\text{ with }v\,\in\,\Sigma(t)}\!\!\!(\sigma(t)-1)\:\le\:
5+4\,(\zeta_S^*-1)+(\beta-1)\:=\:4\zeta_S^*+\beta\:\le\:\tfrac32\,\beta,\]
since $\beta\ge\tfrac23\,(\zeta_S^*)^2\ge8\zeta_S^*$. Since~$v$ has at
least $\bigl(\frac32+\eps)\,\beta$ colours in its list, there exists a free
colour for~$v$, i.e., a colour different from the colour of all the
vertices in conflict with~$v$. We colour~$v$ with such a free colour. By
the construction of~$G_2$ and~$\Sigma_2$, it is easy to verify that this
defines a list $\Sigma$-colouring of~$G$, contradicting the choice of~$G$
as a counterexample.

\subsubsection{The structure from (S3) is present
  in~{\boldmath$G$}}\label{ss-S3}

Let~$X$ and~$Y$ be two non-empty disjoint subsets of~$V$ such that the pair
$(X,Y)$ is a very special $\zeta_S^*$-pair satisfying the condition
of~(S3). We can remove from~$X$ any vertex not adjacent to any vertex
in~$Y$.

\begin{claim}\label{cl2}
  For all $y\in Y$, we have that if $X^y=\{x_1,x_2\}$, then
  $y\in\Sigma(x_1)\cap\Sigma(x_2)$.
\end{claim}

\begin{proof}
  Suppose we have $y\notin\Sigma(x_1)$. Since $(X,Y)$ is special, $y$ has
  degree four, and it has a neighbour~$u$ not in~$X^y$ of degree four. We
  also have $\ds(y)\le4+2\cdot(4-1)+(\sigma(x_2)-1)\le9+\beta$. By
  contracting the edge~$uy$, we can argue similarly to
  Subsection~\ref{ss-S2} (with~$y$ now playing the role of~$v$) to obtain a
  contradiction.
\end{proof}

\noindent
In the remainder of this subsection we describe how to reduce this case to
a list edge-colouring problem. More precisely, we first define a
modification of the original graph~$G$ into a smaller graph~$G_0$ with
vertex set $V\setminus Y$, inheriting a $\Sigma$-system from that of~$G$,
so that the minimality of~$G$ as a counterexample implies that~$G_0$ admits
a $\Sigma$-colouring. This colouring then provides a partial
$\Sigma$-colouring of~$G$, giving a colour to every vertex outside~$Y$. In
order to extend this partial colouring to the whole graph, we define a
multigraph whose edges are indexed by the vertices in~$Y$, so that an
edge-colouring of that multigraph is exactly the extension of the
$\Sigma$-colouring to~$Y$ we are looking for. In the next subsection we
then describe how Kahn's approach to prove that the list chromatic index is
asymptotically equal to the fractional chromatic index, can be used to
conclude the proof of Theorem~\ref{th1}.

To define $G_0$, we divide the vertices of~$Y$ into three parts according
to their number of neighbours outside $X \cup Y$. Let~$Y'$ be the set of
vertices in~$Y$ with no neighbour outside $X\cup Y$. Consider first the
graph $G[V\setminus Y']$ induced on the set of vertices outside~$Y'$. For
each vertex $y\in Y'$, add an edge between its two neighbours
$\{x_1,x_2\}=X^y$, if those are not already joined by an edge, and
remove~$y$ from $\Sigma(x_1)$ and $\Sigma(x_2)$. Also, add~$x_1$ to
$\Sigma(x_2)$ and~$x_2$ to $\Sigma(x_1)$. Note that after these changes,
$\Sigma(x_1)$ and $\Sigma(x_2)$ cannot be larger than before (since, by
Claim~\ref{cl2}, $y\in\Sigma(x_i)$ for $i=1,2$).

For any vertex $y\in Y\setminus Y'$ with a unique neighbour~$u$ outside
$X\cup Y$, contract the edge~$yu$ (remove multiple edges if they appear),
and, by an abuse of the notation, call the new vertex~$u$ again. For the
two vertices~$x_1$ and~$x_2$ in~$X^y$, remove~$y$ from $\Sigma(x_1)$ and
$\Sigma(x_2)$. For the vertex~$u$ itself, let $\Sigma(u)$ be equal to the
set of all its neighbours.

And, finally, for any vertex $y\in Y\setminus Y'$ with exactly two
neighbours~$u$ and~$u'$ outside $X\cup Y$, contract the edge~$yu$ (remove
multiple edges if they appear), and, by an abuse of the notation, call the
new vertex~$u$ again. For the two vertices~$x_1$ and~$x_2$ in~$X^y$,
remove~$y$ from $\Sigma(x_1)$ and $\Sigma(x_2)$. Add~$u$ to $\Sigma(u')$
and remove $y$ from $\Sigma(u')$ (if it was in this set). For the
vertex~$u$ itself, let $\Sigma(u)$ be equal to the set of all its
neighbours. Note that~$u'$ has degree at most four in $G$, hence certainly
$|\Sigma(u')|\le\beta$.

The graph obtained after the modifications described above is denoted
by~$G_0$, and the resulting sets by $\Sigma_0(v)$, $v\in V(G_0)$. Note
that, by our abuse of the notation, $G_0$ has the vertex set
$V_0=V\setminus Y$. Next we observe that a vertex~$u$ of~$G$ outside $X
\cup Y$ that was adjacent to a vertex $y\in Y$ (and hence may have been
involved in one or more contractions) has degree four in~$G$. Since
vertices in~$Y$ have degree four as well, each contraction increases the
degree by at most two. So in~$G_0$, such a vertex~$u$ has degree at most
twelve, hence we certainly have $|\Sigma_0(u)|\le\beta$. By the
construction above, we saw that for every other vertex $v\in V_0$, we also
have $|\Sigma_0(v)|\le|\Sigma(v)|\le\beta$ or $|\Sigma_0(v)|\le
d_{G_0}(v)\le\beta$.

By its construction, $G_0$ is embeddable in~$S$. Also by construction, and
the remarks above, it is easy to verify the following statement.

\begin{claim}\label{cla:ngbeta}
  If $u,v\in V_0$ are adjacent in~$G$, then~$u,v$ are adjacent in~$G_0$. If
  $u,v\in V_0$ and there is a $t\in V$ with $u,v\in\Sigma(t)$, then~$u,v$
  are either adjacent in~$G_0$, or there is a $t_0\in V_0$ with
  $u,v\in\Sigma(t_0)$.
\end{claim}

\noindent
For each vertex $v\in V_0$ set $L_0(v)=L(v)$. Since $Y\ne\varnothing$, by
the minimality of~$G$, the graph~$G_0$ admits a list
$\Sigma_0$-colouring~$c_0$ with respect to the list assignment~$L_0$.

We now transform this colouring into a partial list $\Sigma$-colouring
of~$G$ with respect to the original list assignment~$L$, by just setting
$c(v)=c_0(v)$ for each vertex $v\in V_0=V\setminus Y$. By
Claim~\ref{cla:ngbeta}, this is indeed a good partial $\Sigma$-colouring of
all the vertices of $V\setminus Y$ in~$G$. The difficult part of the proof
is to show that~$c$ can be extended to~$Y$.

By assumption, at the beginning every vertex in~$Y$ has a list of at least
$\bigl(\frac32+\eps\bigr)\,\beta$ available colours. For each vertex~$y$
in~$Y$, let us remove from~$L(y)$ the colours which are forbidden for~$y$
according to the partial $\Sigma$-colouring~$c$ of~$G$. In the worst case,
these forbidden colours are exactly the colours of the vertices of
$V\setminus Y$ at distance at most two from~$y$.

\medskip
Let us define the multigraph~$H$ as follows: $H$ has vertex set~$X$. And
for each vertex $y\in Y$ we add an edge~$e_y$ between the two neighbours
of~$y$ in~$X$ (in other words, between the two vertices in~$X^y$). Note
that this process may produce multiple edges. We associate a list~$L(e_y)$
to~$e_y$ in~$H$ by taking the list of~$y$ obtained after removing the set
of forbidden colours for~$y$ from the original list~$L(y)$.

In what follows, following the usual terminology for multigraphs, we denote
by~$d_H(x)$ the degree of the vertex~$x$ in the multigraph~$H$, i.e., the
number of edges incident with~$x$ in $H$. By Claim~\ref{cl2}, we have
$N_G(x)\cap Y\subseteq\Sigma(x)$ for every $x\in X$, which guarantees
$d_H(x)\le\sigma(x)$.

We now prove the following lemma.

\begin{lemma} \label{lem:ext}\mbox{}\\*
  A list edge-colouring for~$H$, with the list assignment~$L$ defined as
  above, provides an extension of~$c$ to a list $\Sigma$-colouring of~$G$
  by giving to each vertex $y\in Y$ the colour of the edge~$e_y$ in~$H$.
\end{lemma}

\begin{proof}
  This follows since the pair $(X,Y)$ is very special: For every two
  vertices $y,z\in Y$, if~$y$ and~$z$ are adjacent or have a common
  neighbour $w\notin X$, then $X^y=X^z$. This proves that the two vertices
  adjacent in~$Y$ or with a common neighbour not in~$X$ define parallel
  edges in~$H$ and so will have different colours. If two vertices~$y_1$
  and~$y_2$ of~$Y$ have a common neighbour in~$X$, $e_{y_1}$ and~$e_{y_2}$
  will be adjacent in~$H$ and so will get different colours. Since we have
  already removed from the list of vertices in~$Y$ the set of forbidden
  colours (defined by the colours of the vertices in $V\setminus Y$), there
  will be no conflict between the colours of a vertex in~$Y$ and a vertex
  in $V\setminus Y$. We conclude that the edge-colouring of~$H$ will
  provide an extension of~$c$ to a list $\Sigma$-colouring of~$G$.
\end{proof}

\noindent
The following lemma provides a lower bound on the size of~$L(e)$ for the
edges~$e$ in~$H$.

\begin{lemma}\mbox{}\\*
  Let $e=x_1x_2$ be an edge in~$H$. Then we have
  \[|L(e)|\:\ge\:\bigl(\tfrac32+\eps\bigr)\,\beta-
  (\sigma(x_1)-d_H(x_1))-(\sigma(x_2)-d_H(x_2))-10.\]
\end{lemma}

\begin{proof}
  Let~$y$ be the vertex in~$Y$ such that $e=e_y$. By the definition of~$H$,
  $X^y=\{x_1,x_2\}$. Let~$Z$ be the set of vertices in $V\setminus X$
  adjacent to~$y$ in~$G$. Then, since $(X,Y)$ is a special
  $\zeta_S^*$-pair, $|Z|\le2$ and $|N_G(Z)\setminus Y|\le6$. The colours
  that are possibly forbidden for~$y$ are the colours of $\{x_1,x_2\}$,
  plus the colours of vertices in $(Z\cup N_G(Z))\setminus Y$, plus the
  colours of vertices in
  $(\Sigma(x_1)\setminus Y)\cup(\Sigma(x_2)\setminus Y)$ (note that these
  colours all come from the vertices outside~$Y$). The number of vertices
  in these three sets add up to at most
  $10+(\sigma(x_1)-d_H(x_1))+(\sigma(x_2)-d_H(x_2))$. The lemma follows.
\end{proof}

\noindent
We finish this subsection by applying Lemma~\ref{lem1} in order to obtain
information on the density of subgraphs in~$H$, which we will need in the
next subsection. Recall that for all non-empty subsets $W\subseteq X$,
$Y^W$ denotes the set of vertices~$y\in Y$ with $X^y\subseteq W$ (that is,
the set of vertices of~$Y$ having their two neighbours from~$X$ in~$W$).
By~(S3), we have for all non-empty $W\subseteq X$,
\[e_G(W,V\setminus Y)\:\le\:e_G(W,Y\setminus Y^W)+\zeta_S^*\,|W|.\]

This inequality has the following interpretation in~$H$.

\begin{lemma}\label{lem-density}\mbox{}\\*
  For all non-empty subsets $W\subseteq X({}=V(H)\,)$, we have
  \[\sum_{w\,\in\,W}(\sigma(w)-d_H(w))\:\le\:
  e_H(W,X\setminus W)+\zeta_S^*\,|W|.\]
\end{lemma}

\begin{proof}
  First note that
  $\sum\limits_{w\,\in\,W}(\sigma(w)-d_H(w))\le
  \sum\limits_{w\,\in\,W}(d_G(w)-d_H(w))=e_G(W,V\setminus Y)$. We also have
  $e_G(W,Y\setminus Y^W)=e_H(W,X\setminus W)$. Combining these two
  observations with the formula in (S3) immediately gives the required
  inequality.
\end{proof}

\noindent
At this point, our aim will be to apply Kahn's approach to the
multigraph~$H$ with the list assignment~$L$, to prove the existence of a
proper list edge-colouring for~$H$. This is described in the next
subsection.

We summarise the properties we assume are satisfied by the multigraph~$H$
and the list assignment~$L$ of the edges of~$H$. For these conditions we
just consider $\sigma(v)$ as an integer with certain properties, assigned
to each vertex~$v$ of~$H$.

\smallskip
\qiteee{(H1)}For all vertices~$v$ in~$H$, we have
$d_H(v)\le\sigma(v)\le\beta$.

\smallskip
\qiteee{(H2)}For all edges $e=uv$ in~$H$,
$|L(e)|\ge\bigl(\frac32+\eps\bigr)\,\beta-(\sigma(u)-d_H(u))-
(\sigma(v)-d_H(v))-10$.

\smallskip
\qiteee{(H3)}For all non-empty subsets $W\subseteq V(H)$,
$\sum\limits_{w\,\in\,W}(\sigma(w)-d_H(w))\le
e_H(W,V(H)\setminus W)+\zeta_S^*\,|W|$, for some constant~$\zeta_S^*$.

\subsection{The Matching Polytope and Edge-Colourings}\label{sec-mp}

We briefly describe the matching polytope of a multigraph. More about this
subject can be found in~\cite[Chapter~25]{Sch03}.

Let~$H$ be a multigraph with~$m$ edges. Let~$\MM(H)$ be the set of all
matchings of~$H$, including the empty matching. For each $M\in\MM(H)$, let
us define the $m$-dimensional characteristic vector~$\textbf{1}_M$ as
follows: $\textbf{1}_M=(x_e)_{e\in E(H)}$, where $x_e=1$ for an edge
$e\in M$, and $x_e=0$ otherwise. The matching polytope of~$H$, denoted
$\MP(H)$, is the polytope defined by taking the convex hull of all the
vectors~$\textbf{1}_M$ for $M\in\MM(H)$. Also, for any real
number~$\lambda$, we set
$\lambda\,\MP(H)=\{\,\lambda\,x\mid x\in\MP(H)\,\}$.

Edmonds~\cite{Ed65} gave the following characterisation of the matching
polytope.

\begin{theorem}[Edmonds~\cite{Ed65}]\label{edm-mp}\mbox{}\\*
  A vector $\vec{x}=(x_e)$ is in $\MP(H)$ if and only if $x_e\ge0$ for
  all~$x_e$ and the following two types of inequalities are satisfied:

  \smallskip
  \qite{$\bullet$}For all vertices $v\in V(H)$,
  $\sum\limits_{e:\:\text{$v$ incident to $e$}}\!\!\!x_e\le1$;

  \smallskip
  \qite{$\bullet$}for all subsets $W\subseteq V(H)$ with $|W|\ge3$
  and~$|W|$ odd, $\sum\limits_{e\,\in\,E(W)}\!x_e\le\half\,(|W|-1).$
\end{theorem}

\noindent
The significance of the matching polytope and its relation to list
edge-colouring is indicated by the following important result.

\begin{theorem}[Kahn~\cite{Kah00}]\label{kahn-main}\mbox{}\\*
  For all real numbers $\delta,\nu$, $0<\delta<1$ and $\nu>0$, there exists
  a~$\Delta_{\delta,\nu}$ such that for all $\Delta\ge\Delta_{\delta,\nu}$
  the following holds. If~$H$ is a multigraph and~$L$ is a list assignment
  of colours to the edges of~$H$ so that{

    \smallskip
    \qite{$\bullet$}$H$ has maximum degree at most~$\Delta$;

    \smallskip
    \qite{$\bullet$}for all edges $e\in E(H$), $|L(e)|\ge\nu\Delta$;

    \smallskip
    \qite{$\bullet$}the vector $\vec{x}=(x_e)$ with $x_e=\dfrac1{|L(e)|}$
    for all $e\in E(H)$ is an element of $(1-\delta)\,\MP(H)$.

  }\smallskip\noindent
  Then there exists a proper edge-colouring of~$H$ where each edge gets a
  colour from its own list.
\end{theorem}

\noindent
The theorem above is actually not explicitly stated this way
in~\cite{Kah00}, but can be obtained from the appropriate parts of that
paper. We give some further details about this in the final section of this
paper.

The next lemma allows us to use Theorem~\ref{kahn-main} to complete the
proof.

\begin{lemma}\label{lem-mp}\mbox{}\\*
  Let~$\beta$ and~$\zeta$ be positive real numbers. Let~$H$ be a multigraph
  with a map $\sigma:V(H)\rightarrow\mathbb{N}$, and a weighting
  $(b_e)_{e\in E(H)}$ of the edges with positive real numbers satisfying
  the following three conditions:{

    \smallskip
    \qiteee{(H1')}For all vertices~$v$ in~$H$,
    $d_H(v)\le\sigma(v)\le\beta$.

    \smallskip
    \qiteee{(H2')}For all edges $e=uv$ in~$H$,
    $b_e\ge\bigl(\frac32\,\beta+\tfrac92\,\zeta\bigr)-(\sigma(u)-d_H(u))-
    (\sigma(v)-d_H(v))$.

    \smallskip
    \qiteee{(H3')}For all non-empty $W\subseteq V(H)$,
    $\sum\limits_{w\,\in\,W}(\sigma(w)-d_H(w))\le
    e_H(W,V(H)\setminus W)+\zeta\,|W|$.

  }\noindent
  Then for all edges $e\in E(H)$, we have $b_e\ge\half\,\beta$. And the
  vector $(1/b_e)_{e\in E(H)}$ is in~$\MP(H)$.
\end{lemma}

\noindent
The proof of Lemma~\ref{lem-mp} will be given in
Subsection~\ref{proof-lem-mp}. This lemma guarantees that for any $\eps>0$,
there exists a~$\beta_\eps$ such that for all $\beta\ge\beta_\eps$,
Theorem~\ref{kahn-main} can be applied to a multigraph~$H$ with an edge
list assignment~$L$ satisfying properties (H1)\,--\,(H3) stated at the end
of the previous subsection.

To see this, take $\delta_\eps=\dfrac\eps{3+2\eps}$, so $0<\delta_\eps<1$.
In order to be able to apply Theorem~\ref{kahn-main}, we want to prove the
existence of~$\beta_{\eps,\zeta_S^*}$ such that for any
$\beta\ge\beta_{\eps,\zeta_S^*}$, the vector $\vec{x} =(x_e)$,
$x_e=\dfrac1{|L(e)|}$, is in $(1-\delta_\eps)\,\MP(H)$. Let~$\zeta_S^*$ be
the constant described in condition~(H3). By condition~(H2), we have for
all $e=uv$ in~$H$,
\begin{align*}
  (1-\delta_\eps)\,|L(e)|\:&\ge\:
  (1-\delta_\eps)\,\bigl(\bigl(\tfrac32+\eps\bigr)\,\beta-
  (\sigma(u)-d_H(u))-(\sigma(v)-d_H(v))-10\bigr)\\
  &\ge\:(1-\delta_\eps)\,\bigl(\tfrac32+\eps\bigr)\,\beta-
  (\sigma(u)-d_H(u))-(\sigma(v)-d_H(v))-10\\
  &=\:\bigl(\tfrac32\,\beta+\half\,\eps \beta\bigr)-(\sigma(u)-d_H(u))-
  (\sigma(v)-d_H(v))-10.
\end{align*}
Let $\beta_{\eps,\zeta_S^*}=\dfrac{9\zeta_S^*+20}{\eps}$. For
$\beta\ge\beta_{\eps,\zeta_S^*}$, we have
\[(1-\delta_\eps)\,|L(e)|\:\ge\:
\bigl(\tfrac32\,\beta+\tfrac92\, \zeta_S^*\bigr)-(\sigma(u)-d_H(u))-
(\sigma(v)-d_H(v)).\]
So by Lemma~\ref{lem-mp}, taking $b_e=(1-\delta_\eps)\,|L(e)|$, the vector
$\Bigl(\dfrac{x_e}{1-\delta_\eps}\Bigr)_{e\in E(H)}$ is in~$\MP(H)$. We
infer that $\vec{x}\in(1-\delta_\eps)\,\MP(H)$.

Notice that the first conclusion of Lemma~\ref{lem-mp} means that
$|L(e)|=\dfrac{b_e}{1-\delta_\eps}>b_e\ge\tfrac12\,\beta$ for all edges~$e$
in~$H$.

Now set
$\beta_{S,\eps}=\max\{\,\beta_S^*,\,\beta_{\eps,\zeta_S^*},\,
\Delta_{\delta_\eps,1/2}\,\}$, where $\zeta_S^*=132\,(3-\chi(S))$,
$\beta_S^*=\tfrac23\,(\zeta_S^*)^2$ (see Lemma~\ref{lem1} and the text
after it), $\beta_{\eps,\zeta_S^*}=\dfrac{9\zeta_S^*+20}{\eps}$,
$\delta_\eps=\dfrac\eps{3+2\eps}$ (see above),
and~$\Delta_{\delta_\eps,1/2}$ is according to Theorem~\ref{kahn-main}
(where $\Delta_{\delta_\eps,1/2}$ is chosen since we have
$|L(e)|\ge\tfrac12\,\beta$ for all $e\in E(H)$ (see above)). Assume
$\beta\ge\beta_{S,\eps}$. Then we can apply Theorem~\ref{kahn-main}, which
implies that the multigraph~$H$ defined in Subsection~\ref{sec2.1} has a
list edge-colouring corresponding to the list assignment~$L$.
Lemma~\ref{lem:ext} then implies that the colouring~$c$ can be extended to
a list \mbox{$\Sigma$-colouring} of the original graph~$G$. This final
contradiction completes the proof of Theorem~\ref{th1}.\eop

\section{Proofs of the Main Lemmas}\label{thelemmas}

We use the terminology and notation from the previous sections.

\subsection{Proof of Lemma~\ref{lem1}}\label{proof-lem1}

Let~$S$ be a surface, set $\zeta_S^*=132\,(3-\chi(S))$, and let~$G$ be a
graph embeddable in~$S$, so that~$G$ is edge-maximal with respect to being
embeddable in~$S$.

From Lemma~\ref{lem0}, we immediately obtain the following.

\begin{claim}\label{cl-a}
  For any vertex~$v$ and any two consecutive neighbours~$u_1,u_2$ of~$v$
  (consecutive with respect to the chosen circular order imposed by the
  embedding), we have $u_1u_2\in E(G)$.
\end{claim}

\noindent
Next we prove that we can assume~$G$ has a cellular embedding in~$S$.
If~$G$ is a tree, then every leaf will give a structure from~(S2). So we
can assume~$G$ is not a tree. Assume~$S$ is orientable with genus~$h$. By
the definition of $\og(G)$, we must have $\og(G)\le h$, and hence
$\chi(\OS_{\og(G)})=2-2\,\og(G)\ge2-2h=\chi(S)$. That also means that the
constant in Lemma~\ref{lem1} satisfies
$\zeta_{\OS_{\og(G)}}^*\le\zeta_S^*$. Hence if we prove the lemma
assuming~$G$ is embeddable in~$\OS_{\og(G)}$, then the lemma for~$G$
embeddable in~$S$ directly follows. So we can use Lemma~\ref{lem1} with the
surface~$\OS_{\og(G)}$ instead of~$S$, and by Lemma~\ref{lem-embedding}, we
can use a cellular embedding of~$G$ in~$\OS_{\og(G)}$.

If~$S$ is non-orientable, then exactly the same argument can be applied,
this time using the surface~$\NOS_{\nog(G)}$ (and using the assumption
that~$G$ is not a tree).

\medskip
We need some further notation and terminology. The set of faces of~$G$ is
denoted by~$F$. Recall that since the embedding in~$S$ is cellular, every
face is homeomorphic to an open disk in~$\mathbb{R}^2$. For such a
face~$f$, a \emph{boundary walk of~$f$} is a walk consisting of vertices
and edges as they are encountered when walking along the whole boundary
of~$f$, starting at some vertex. The \emph{degree of a face~$f$},
denoted~$d(f)$, is the number of edges on the boundary walk of~$f$. Note
that this means that some edges may be counted more than once. The
\emph{order} of a face is the number of vertices in its boundary. We always
have that the order of~$f$ is at most~$d(f)$.

\medskip
Now suppose that~$G$ does not contain any of the structures~(S1) or~(S2).
In order to prove Lemma~\ref{lem1}, we only need to prove that~$G$ contains
structure~(S3). In other words, we need to prove that~$G$ contains a very
special $\zeta_S^*$-pair $(X,Y)$ with~$X$ and~$Y$ non-empty which satisfies
the inequality of~(S3) for all non-empty subsets $W\subseteq X$.

We easily see that~$G$ has at least $\zeta_S^*+2\ge134$ vertices (otherwise
it contains structure~(S1)). So by Lemma~\ref{lem-1} we know that all
vertices have degree at least three.

Let us call the vertices of degree at least $\zeta_S^*+1$ \emph{big}; the
other vertices are called \emph{small}. We use~$B$ to denote the set of big
vertices.

Since we assumed that~$G$ does not contain structure~(S2), we immediately
get:

\begin{claim}\label{cl-d}
  All vertices of degree at most five have at least two big neighbours.
\end{claim}

\noindent
We continue our analysis using the classical technique of
\emph{discharging} (see, e.g., \cite[Section 15.2]{BoMu08}). Give each
vertex~$v$ an initial charge $\rho(v)=6d(v)-36$. Since~$G$ is simple and
has a cellular embedding in~$S$, every face has degree at least three. This
gives $2\,|E|\ge3\,|F|$, and hence, by Euler's Formula,
$\sum\limits_{v\,\in\,V}\rho(v)=12\,|E|-36\,|V|\le-36\,|V|+36\,|E|-
36\,|F|=-36\,\chi(S)$.

We further redistribute charges according to the following rules:{

  \smallskip
  \qiteee{(R1)}Each vertex of degree three that is adjacent to three big
  vertices receives a charge~6 from each of its neighbours.

  \smallskip
  \qiteee{(R2)}Each vertex of degree three that is adjacent to two big
  vertices receives a charge~9 from each of its big neighbours.

  \smallskip
  \qiteee{(R3)}Each vertex of degree four that is adjacent to four big
  vertices receives a charge~3 from each of its big neighbours.

  \smallskip
  \qiteee{(R4)}Each vertex of degree four that is adjacent to three big
  vertices receives a charge~4 from each of its big neighbours.

  \smallskip
  \qiteee{(R5)}Each vertex of degree four that is adjacent to two big
  vertices receives a charge~6 from each of its big neighbours.

  \smallskip
  \qiteee{(R6)}Each vertex of degree five receives a charge~3 from each of
  its big neighbours.

}\medskip\noindent
Denote the resulting charge of a vertex $v\in V$ after applying rules
(R1)\,--\,(R6) by $\rho'(v)$. Since the global charge has been preserved,
we have $\sum\limits_{v\,\in\, V}\rho'(v)\le-36\,\chi(S)$. We will show
that for most $v\in V$, $\rho'(v)$ is non-negative.

Combining Claim~\ref{cl-d} with rules (R1)\,--\,(R6) and our knowledge that
$\rho(v)=6d(v)-36$, we find that $\rho'(v)=0$ if $d(v)=3,4$, while
$\rho'(v)\ge0$ if $d(v)=5$. If~$v$ is a small vertex with $d(v)\ge6$, we
have $\rho'(v)=\rho(v)=6d(v)-36\ge0$.

It follows that we must have
\begin{equation}\label{eq7}
  \sum_{v\,\in\,B}\rho'(v)\:\le\:-36\,\chi(S).
\end{equation}
To derive the relevant consequence of that formula, we must make a detailed
analysis of the neighbours of vertices in~$B$.

As we explained in Subsection~\ref{gr-on-sur}, the embedding of~$G$ in~$S$
allows us to choose a circular order on the neighbours of each vertex~$v$.
By Claim~\ref{cl-a} we know that two consecutive vertices in this order are
adjacent. If~$u$ is a neighbour of~$v$, then by $u^+,u^{++}$ we denote the
successor and second successor of~$u$ in the circular order of neighbours
of~$v$, while $u^-,u^{--}$ denote the predecessor and second predecessor
of~$u$ in that order.

\medskip
We distinguish five different types of neighbours of a vertex $v\in B$:
\begin{align*}
  M_1(v)\:&=\:\{\,u\in N(v)\mid
  \{u^-,u^{--},u^+,u^{++}\}\cap B\ne\varnothing\,\};\\
  M_{4a}(v)\:&=\:\{\,u\in N(v)\setminus M_1(v)\mid\text{$d(u)=4$ and
    $u^-$ or $u^+$ have degree at least five}\,\};\\
  M_{4b}(v)\:&=\:\{\,u\in N(v)\setminus M_1(v)\mid
  d(u)=d(u^-)=d(u^+)=4\,\};\\
  M_5(v)\:&=\:\{\,u\in N(v)\setminus M_1(v)\mid d(u)=5\,\};\\
  M_6(v)\:&=\:\{\,u\in N(v)\setminus M_1(v)\mid d(u)\ge6\,\}.
\end{align*}

First observe that if a neighbour~$u$ of~$v$ has degree three, then~$u^-$
or~$u^+$ is in~$B$. This follows since by Claim~\ref{cl-a}, $u^-$ and~$u^+$
are also neighbours of~$u$. And by Claim~\ref{cl-d}, a vertex of degree
three has at least two big neighbours. From this observation we also get
that if $u\in N(v)\setminus M_1(v)$ is a small vertex, then~$u^-$ and~$u^+$
both have degree at least four.

As a consequence, every neighbour of~$v$ is in exactly one set. Our aim in
the following, in order to prove Lemma~\ref{lem1}, is to show that most
neighbours of vertices $v\in B$ are in~$M_{4b}(v)$.

\medskip
We now evaluate the charge that a vertex $v\in B$ has given to its
neighbours. If $u\in M_1(v)$, then~$v$ gave at most $9+9+9=27$ to
$\{u^-,u,u^+\}$; if $u\in M_{4a}(v)$, then~$v$ gave at most $3+6+6=15$ to
$\{u^-,u,u^+\}$; if $u\in M_{4b}(v)$, then~$v$ gave at most $6+6+6=18$ to
$\{u^-,u,u^+\}$; if $u\in M_5(v)$, then~$v$ gave at most $6+3+6=15$ to
$\{u^-,u,u^+\}$; and, finally, if $u\in M_6(v)$, then~$v$ gave at most
$6+0+6=12$ to $\{u^-,u,u^+\}$. Setting $m_1=|M_1(v)|$,
$m_{4a}=|M_{4a}(v)|$, $m_{4b}=|M_{4b}(v)|$, $m_5=|M_5(v)|$, and
$m_6=|M_6(v)|$, we can conclude that~$v$ gave at most
\begin{align*}
  \tfrac13\,(27m_1+{}&15m_{4a}+18m_{4b}+15m_5+12m_6)\\
  &{}\le\:9m_1+6m_{4b}+5(m_{4a}+m_5 +m_6)\:\le\:5d(v)+4m_1+m_{4b}
\end{align*}
to its neighbourhood. This means that the remaining charge $\rho'(v)$ of a
vertex $v\in B$ must satisfy
\[\rho'(v)\:\ge\:(6d(v)-36)-(5d(v)+4m_1+m_{4b})\:=\:d(v)-m_{4b}-4m_1-36.\]
By definition, $|M_1(v)|$ is at most four times the number of neighbours
of~$v$ in~$B$. Consider the subgraph $G[B]$ of~$G$ induced by~$B$. As a
subgraph of~$G$, this graph is embeddable in~$S$. Given that it is simple
as well, no face of such an embedding is incident with two or fewer edges.
So Euler's Formula means that $G[B]$ has at most $3\,|B|-3\,\chi(S)$ edges,
and hence
\[\sum_{v\,\in\,B}|M_1(v)|\:\le\:\sum_{v\,\in\,B}4d_{G[B]}(v)\:=\:
8\,|E(G[B])|\:\le\:24\,|B|-24\,\chi(S).\]
Combining the last two inequalities with~\eqref{eq7} gives
\[-36\,\chi(S)\:\ge\:\sum_{v\,\in\,B}\rho'(v)\:\ge\:
\sum_{v\,\in\,B}(d(v)-|M_{4b}(v)|)-4\,(24\,|B|-24\,\chi(S))-36\,|B|.\]
Using that $B\ne\varnothing$ (otherwise~$G$ contains structure~(S1)) and
$\chi(S)\le2$, this can be rewritten as
\[\sum_{v\,\in\,B}(d(v)-|M_{4b}(v)|)\:\le\:132\,|B|-132\,\chi(S)\:<
\:132\,|B|+132\,(2-\chi(S))\:\le\:132\,(3-\chi(S))\,|B|.\]

Define $X_0=B$ and $Y_0=\bigcup_{v\in B}M_{4b}(v)$. Note that the previous
inequality can be written
\begin{equation}\label{eq4}
  e(X_0,V\setminus Y_0)\:<\:\zeta_S^*\,|X_0|.
\end{equation}
Also observe that the pair $(X_0,Y_0)$ is a special $\zeta_S^*$-pair: The
vertices in~$X_0$ are the big vertices, hence have degree at least
$\zeta_S^*+1$. For all vertices $u\in Y_0$, we have $u\in M_{4b}(v)$ for
some $v\in B$, and hence $u$, $u^-$ and~$u^+$ have degree four in~$G$, and
the fourth neighbour of~$u$ is in $B=X_0$ by Claim~\ref{cl-d}.

\smallskip
We need some more information about the neighbours of vertices in~$Y_0$.

\begin{claim}\label{y0}
  Let~$v$ be a big vertex, $u\in M_{4b}(v)$, and~$w$ be the big neighbour
  of~$u$ different from~$v$. Then all of $vu^+$, $vu^-$, $wu^+$ and~$wu^-$
  are edges of~$G$.
\end{claim}

\begin{proof}
  Consider the circular order of the neighbours of~$u$ imposed by the
  embedding. In any circular order different from $(v,u^+,w,u^-)$ or the
  reverse, $u^+$ and~$u^-$ are consecutive. By Claim~\ref{cl-a}, this means
  that $u^+u^-\in E$. So the neighbours of~$u^-$ are $\{v,u,u^+,u^{--}\}$.
  Since $u,u^+,u^{--}\notin B$, by the definition of~$M_{4b}(v)$, that
  means~$u^-$ has only one big neighbour, contradicting Claim~\ref{cl-d}.

  So the only possible circular orders are $(v,u^+,w,u^-)$ or the reverse,
  and the result follows by Claim~\ref{cl-a}.
\end{proof}

\noindent
Using Claim~\ref{y0}, it follows easily that if $y,z\in Y_0$ are adjacent,
then $X_0^y=X_0^z$; while if $y,z\in Y_0$ share a neighbour $u\notin X_0$,
then~$u$ has degree four and its two neighbours distinct from~$y$ and~$z$
are in~$X_0^y$ and in~$X_0^z$. This gives $X_0^y=X_0^z$.

Thus, we have shown that the pair $(X_0,Y_0)$ is very special.

\medskip
Since~$X_0$ and~$Y_0$ are non-empty, we are done if the pair $(X_0,Y_0)$
also satisfies the inequalities of~(S3) for any non-empty subset
$W\subseteq X_0$. Suppose this is not the case. So there must exist a set
$Z_1\subseteq X_0$ with
\[e(Z_1,V\setminus Y_0)\:>\:
e(Z_1,Y_0\setminus Y_0^{Z_1})+\zeta_S^*\,|Z_1|.\]
Define $X_1=X_0\setminus Z_1$ and $Y_1=Y_0^{X_1}$. Again, by construction,
it is easy to see that $(X_1,Y_1)$ is a very special $\zeta_S^*$-pair. If
it does not satisfy condition~(S3), we iterate the process (see
Figure~\ref{fig:density}) and eventually obtain a very special
$\zeta_S^*$-pair $(X_k,Y_k)$ satisfying condition~(S3). To conclude the
proof, we only need to check that $X_k$ and $Y_k$ are non-empty.

\begin{figure}[htbp]
  \begin{center}
    \includegraphics[scale=0.6]{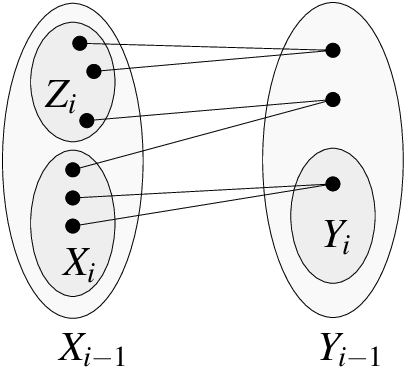}
    \caption{$X_i=X_{i-1}\setminus Z_i$ and $Y_i=Y_{i-1}^{X_i}$.
      \label{fig:density}}
  \end{center}
\end{figure}

Let $1\le i\le k$. Since $X_i=X_{i-1}\setminus Z_i$, we have
\begin{align*}
  e(X_i,V\setminus Y_i)\:&=\:e(X_{i-1},V\setminus Y_i)-
  e(Z_i,V \setminus Y_i)\\
  &\hskip-5mm=\:e(X_{i-1},V\setminus Y_{i-1})+
  e(X_{i-1},Y_{i-1}\setminus Y_i)-e(Z_i,V\setminus Y_{i-1})-
  e(Z_i,Y_{i-1}\setminus Y_i)\\
  &\hskip-5mm=\:e(X_{i-1},V\setminus Y_{i-1})-e(Z_i,V\setminus Y_{i-1})+
  e(X_i,Y_{i-1}\setminus Y_i).
\end{align*}
Since $Y_i=Y_{i-1}^{X_i}$, every neighbour $u\in Y_{i-1}\setminus Y_i$ of a
vertex in~$X_i$ has exactly one neighbour in~$Z_i$ (see
Figure~\ref{fig:density}). Hence,
$e(X_i,Y_{i-1}\setminus Y_i)=e(Z_i,Y_{i-1}\setminus Y_{i-1}^{Z_i})$. So we
have
\[e(X_{i-1},V\setminus Y_{i-1})\:=\:e(X_i,V\setminus Y_i)+
e(Z_i,V\setminus Y_{i-1})-e(Z_i,Y_{i-1}\setminus Y_{i-1}^{Z_i}).\]
By the definition of~$Z_i$, we have
$e(Z_i,V\setminus Y_{i-1})>e(Z_i,Y_{i-1}\setminus Y_{i-1}^{Z_i})+
\zeta_S^*\,|Z_i|$. Combining the last two expressions gives
\[e(X_{i-1},V\setminus Y_{i-1})\:>\:e(X_i,V\setminus
Y_i)+\zeta_S^*\,|Z_i|.\]
Setting $Z^*=\bigcup\limits_{1\le i\le k}Z_i$, we have
$e(X_k,V\setminus Y_k)<e(X_0,V\setminus Y_0)-\zeta_S^*\,|Z^*|$. As a
consequence, using~\eqref{eq4},
\[|Z^*|\:<\:\frac{e(X_0,V\setminus Y_0)-e(X_k,V\setminus Y_k)}{\zeta_S^*}
\:\le\:\frac{e(X_0,V\setminus Y_0)}{\zeta_S^*}\:<\:
\frac{\zeta_S^*\,|X_0|}{\zeta_S^*}\:=\:|X_0|.\]
Since $X_k=X_0\setminus Z^*$, this implies $|X_k|>0$, which leads to
$X_k\ne\varnothing$.

Finally, let $v\in X_k\ne\varnothing$ and assume $Y_k=\varnothing$. Taking
$W=\{v\}$ in the inequality in~(S3) (which by construction is satisfied by
$(X_k,Y_k)$), we obtain $d(v)\le\zeta_S^*$. Since~$v$ is a big vertex,
$d(v)\ge\zeta_S^*+1$. This contradiction means that we must have
$Y_k\ne\varnothing$, which concludes the proof of Lemma~\ref{lem1}.\eop

\subsection{Proof of Lemma~\ref{lem-mp}}\label{proof-lem-mp}

We recall the hypotheses of the lemma: We have positive real
numbers~$\beta$ and $\zeta$; $H$ is a multigraph; each vertex~$v$ of~$H$
has an associated integer~$\sigma(v)$; and for each edge~$e$ a positive
real number~$b_e$ is given. In this subsection, all degrees~$d(v)$ are in
the multigraph~$H$.

The following three conditions are satisfied:{

  \smallskip
  \qiteee{(H1')}For all vertices~$v$ in~$H$, $d(v)\le\sigma(v)\le\beta$.

  \smallskip \qiteee{(H2')}For all edges $e=uv$ in~$H$,
  $b_e\ge\bigl(\frac32\,\beta+\frac92\,\zeta)-(\sigma(u)-d(u))-
  (\sigma(v)-d(v))$.

  \smallskip
  \qiteee{(H3')}For all non-empty subsets $W\subseteq V(H)$,
  $\sum\limits_{w\,\in\,W}(\sigma(w)-d(w))\le
  e_H(W,V(H)\setminus W)+\zeta\,|W|$.

}\medskip\noindent
In the proof that follows, we will show that the vector $\vec{x}=(x_e)$,
$x_e=1/b_e$, is in~$\MP(H)$.

For an edge $e=uv$ in~$H$, define
\begin{equation}\label{eq2}
  a_e\:=\:\bigl(\tfrac32\,\beta+\tfrac92\,\zeta\bigr)-(\sigma(u)-d(u))-
  (\sigma(v)-d(v))\qquad\text{and}\qquad y_e\:=\:\frac{1}{a_e}.
\end{equation}
We will in fact prove that the vector $\vec{y}=(y_e)$ is in the matching
polytope~$\MP(H)$. Since $b_e\ge a_e$, we have $x_e=1/b_e\le1/a_e=y_e$. So,
by Edmonds' characterisation of the matching polytope, if
$\vec{y}\in\MP(H)$, this guarantees that $\vec{x}\in\,\MP(H)$, as required.

Applying condition~(H3') to the set $W=\{v\}$ gives
$\sigma(v)-d(v)\le d(v)+\zeta$, which implies:{

  \smallskip
  \qitee{(a)}\emph{For all vertices $v\in V(H)$, we have
    $d(v)\ge\half\,(\sigma(v)-\zeta)$.}

}\smallskip\noindent
Let $e=uv$ be an edge of~$H$. If we use the estimate above for both~$u$
and~$v$ in the definition of~$a_e$ in~\eqref{eq2}, recalling that
$\sigma(u),\sigma(v)\le\beta$, we obtain
\[a_e\:\ge\:\tfrac32\,\beta+\tfrac92\,\zeta-\half\,\sigma(u)-
\half\,\sigma(v)-\zeta\:\ge\:\half\beta+\tfrac72\,\zeta.\]
On the other hand, if we use observation~(a) for~$u$ only, we get
\[a_e\:\ge\:d(v)+\tfrac32\,\beta+\tfrac92\,\zeta-\half\,\sigma(u)-
\sigma(v)-\half\,\zeta\:\ge\:d(v)+4\zeta.\]
Hence, the following two conclusions hold.{

  \smallskip
  \qitee{(b)}\emph{For all edges $e=uv$ in~$E(H)$, we have
    $a_e\ge d(v)+4\zeta$.}

  \smallskip
  \qitee{(c)}\emph{For all edges $e\in E(H)$, we have
    $a_e\ge\half\,\beta+\tfrac72\,\zeta$.}

}\smallskip
Note that observation~(c) also gives $b_e\ge a_e\ge\half\,\beta$ for all
$e\in E(H)$, as required.

By observation~(b), we find, since $\zeta>0$,
\[\sum_{e\,\ni\,v}\frac1{a_e}\:\le\:d(v)\cdot\frac1{d(v)+4\zeta}\:<\:1,\]
which shows that

\begin{claim}\label{condition-1}
  For all vertices $v\in V(H)$, we have
  $\displaystyle\sum_{e\,\ni\,v}y_e<1$.
\end{claim}

\noindent
Using Theorem~\ref{edm-mp}, all that remains is to prove that for all
$W\subseteq V(H)$ with $|W|\ge3$ and~$|W|$ odd, we have
$\sum\limits_{e\,\in\,E(W)}y_e\le\half\,(|W|-1)$. We will actually prove
this for all $|W|\ge3$. Note that we can certainly assume
$E(W)\ne\varnothing$.

\medskip
Using observation~(b), we infer that
\[\sum_{e\,\in\,E(W)}\frac1{a_e}\:\le\:
\half\sum_{u\,\in\,W}\frac{d_{H[W]}(u)}{d(u)+4\zeta}\:=\:
\half\sum_{u\,\in\,W}\Bigl(\frac{d(u)}{d(u)+4\zeta}-
\frac{d(u)-d_{H[W]}(u)}{d(u)+4\zeta}\Bigr).\]
Since $\dfrac{d(u)}{d(u)+4\zeta}\le\dfrac{\beta}{\beta+4\zeta}$ and
$\dfrac{d(u)-d_{H[W]}(u)}{d(u)+4\zeta}\ge
\dfrac{d(u)-d_{H[W]}(u)}{\beta+4\zeta}$, this implies
\[\sum_{e\,\in\,E(W)}\frac1{a_e}\:\le\:
\half\,|W|\,\frac{\beta}{\beta+4\zeta}
-\half\,\frac{e(W,W^c)}{\beta+4\zeta}.\]
Here we used that
$\sum\limits_{u\,\in\,W}\bigl(d(u)-d_{H[W]}(u)\bigr)=e(W,W^c)$, where
$W^c=V(H)\setminus W$.

If $e(W,W^c)\ge\beta$, we obtain, since $\zeta>0$,
\[\sum_{e\,\in\,E(W)}\!y_e\:\le\:
\half\,(|W|-1)\cdot\frac{\beta}{\beta+4\zeta}\:<\:\half\,(|W|-1).\]

So we can assume in the following that $e(W,W^c)\le\beta$, in which case
Condition~(H3') of Lemma~\ref{lem-mp} implies
\[\sum_{u\,\in\,W}(\sigma(u)-d(u))\:\le\:e(W,W^c)+\zeta\,|W|\:\le\:
\beta+\zeta\,|W|.\]
For a vertex~$u$ set $c(u)=\sigma(u)-d(u)$, and for a set of vertices~$U$
define $c(U)=\sum\limits_{u\,\in\,U}c(u)$. So we can write the inequality
above as $c(W)\le\beta+\zeta\,|W|$.

In the following we use the fact that all~$a_e$ are large enough to find a
bound for the sum $\sum\limits_{e\,\in\,E(W)}\!a_e^{-1}$. To this aim,
recall from~\eqref{eq2} that
$a_e=\bigl(\frac32\,\beta+\tfrac92\,\zeta\bigr)-c(u)-c(v)$ for all edges
$e=uv$ in~$H$. This gives
\[\sum_{e\,\in\,E(W)}\!a_e\:=\:
\bigl(\tfrac32\,\beta+\tfrac92\,\zeta\bigr)\,|E(W)|-
\sum_{u\,\in\,W}c(u)\,d_{H[W]}(u).\]
Since $d_{H[W]}(u)\le d(u)=\sigma(u)-c(u)\le\beta-c(u)$, we have
\[\sum_{e\,\in\,E(W)}\!a_e\:\ge\:\bigl(\tfrac32\,\beta+
\tfrac92\,\zeta\bigr)\,|E(W)| -\beta\,c(W)+\sum_{u\,\in\,W}c(u)^2.\]

Set $q=\frac32\,\beta+\tfrac92\,\zeta$ and
$p=\min\limits_{uv\in E(W)} \bigl\{q-c(u)-c(v)\bigr\}$. This means that
$q-p=\max\limits_{uv\in E(W)}\bigl\{c(u)+c(v)\bigr\}$. Let $e=uv$ be an
edge in~$E(W)$ so that $c(u)+c(v)=q-p$. Then
$c(u)^2+c(v)^2\ge\half\,(q-p)^2$, and hence we can be sure that
\[\sum_{e\,\in\,E(W)}\!a_e\:\ge\:q\,|E(W)|-\beta\,c(W)+\half\,(q-p)^2.\]

We now use this inequality and the following claim to bound
$\sum\limits_{e\,\in\,E(W)}\!a_e^{-1}$.

\begin{claim}\label{estimate}
  Let $r_1,\ldots,r_m$ be~$m$ real numbers such that
  $0<p\le r_1,\dots,r_m\le q$ and
  $\sum\limits_{1\,\le\,i\,\le\,m}\!r_i\ge q\,m-(q-p)\,S$, for some
  $S\ge0$. Then we have
  $\sum\limits_{1\,\le\,i\,\le\,m}\!r_i^{-1}\le\dfrac{S}p+\dfrac{m-S}q$.
\end{claim}

\begin{trivlist}\item[]{\bf Proof}\mbox{ \ }The result is trivial if $p=q$,
  so suppose $p<q$. For any $1\le i\le m$, set $c_i=\dfrac{q-r_i}{q-p}$.
  Now we have $0\le c_i\le1$ for all $1\le i\le m$, and
  $\sum\limits_{1\,\le\,i\,\le\,m}\!c_i\le S$. Since the function
  $x\mapsto\dfrac1x$ is convex, we have that for $1\le i\le m$,
  \[\frac1{r_i}\:=\:\frac1{q-c_i\,(q-p)}\:=\:\frac1{c_i\,p+(1-c_i)\,q}
  \:\le\:c_i\,\frac1p+(1-c_i)\,\frac1q\:=\:
  c_i\,\Bigl(\frac1p-\frac1q\Bigr)+\frac1q.\]
  As a consequence,
  \[\sum_{1\,\le\,i\,\le\,m}\frac1{r_i}\:\le\:\Bigl(\frac1p-\frac1q\Bigr)
  \sum_{1\,\le\,i\,\le\,m}c_i+\frac{m}{q}\:\le\:
  \Bigl(\frac1p-\frac1q\Bigr)\,S+\frac{m}{q}\:\le\:\frac{S}p+\frac{m-S}q.\]

  \vspace{-7mm}\qquad\hspace*{\fill}$\Box$

  \vspace{5mm}\end{trivlist}

\noindent
We set $R=\beta\,c(W)-\half\,(q-p)^2$ and $S=\dfrac{R}{q-p}$. Using
Claim~\ref{estimate}, at this point we have
\[\sum_{e\,\in\,E(W)}\frac1{a_e}\:\le\:\frac{S}p+\frac{|E(W)|-S}{q}\:=\:
\frac{S\,(q-p)}{pq}+\frac{|E(W)|}{q}\:=\:
\frac{R}{pq}+\frac{2\,|E(W)|}{3\beta+9\zeta}.\]
Notice that by condition~(H3') of Lemma~\ref{lem-mp},
$2\,|E(W)|\le\sum\limits_{u\,\in\,W}\sigma(u)-2c(W)+\zeta\,|W|\le
\beta\,|W|-2c(W)+\zeta\,|W|$. Hence we find
\begin{equation}\label{eq3}
  \sum_{e\,\in\,E(W)}\frac1{a_e}\:\le\:
  \frac{\beta\,|W|}{3\beta+9\zeta}+\frac{R}{pq}-
  \frac{2c(W)}{3\beta+9\zeta}+\frac{\zeta\,|W|}{3\beta+9\zeta}.
\end{equation}

\begin{claim}\label{Rpq}
  We have
  $\dfrac{R}{pq}-\dfrac{2c(W)}{3\beta+9\zeta}+
  \dfrac{\zeta\,|W|}{3\beta+9\zeta}\le\dfrac{\zeta}{\beta+3\zeta}\,|W|$.
\end{claim}

\begin{proof}
  Since $q=\frac32\,\beta+\tfrac92\,\zeta$, we only have to prove that
  $\dfrac{2R}p-2c(W)\le2\zeta\,|W|$.

  Let us write $q-p=\alpha\beta$, and so
  $p=\half\,(3-2\alpha)\,\beta+\tfrac92\,\zeta$ and
  $R=\beta\,c(W)-\half\,\alpha^2\,\beta^2$. We have
  \[\frac{2R}{p}-2c(W)\:=\:
  \frac{2\beta\,c(W)}p-\frac{\alpha^2\,\beta^2}p-2c(W).\]

  If $p\ge\beta$, this expression is negative, so we can assume that
  $p<\beta$. In this case, using that $c(W)\le\beta+\zeta\,|W|$, we have
  \begin{align*}
    \frac{2R}{p}-2c(W)&\:=\:
    \frac{2\beta\,c(W)}p-\frac{\alpha^2\,\beta^2}p-2c(W)\\
    &\:=\:2c(W)\,\frac{\beta-p}p-\frac{\alpha^2\,\beta^2}p\:\le\:
    \frac{\beta}p\,(2\beta-2p-\alpha^2\,\beta)+
    2\zeta\,|W|\,\frac{\beta-p}p.
  \end{align*}
  As $2p=(3-2\alpha)\,\beta+9\zeta$, we have
  $2\beta-2p-\alpha^2\,\beta=(-1+2\alpha-\alpha^2)\,\beta-9\zeta=
  -(\alpha-1)^2\,\beta-9\zeta<0$. Since $p=a_e$ for some edge~$e$, we have
  $p\ge\half\,\beta$ by observation~(c). Hence, $(\beta-p)/p\le1$ and we
  can conclude that $2R/p-2c(W)\le2\zeta\,|W|$, which completes the proof
  of the claim.
\end{proof}

\noindent
Combining~\eqref{eq3} and Claim~\ref{Rpq}, we obtain
\[\sum_{e\,\in\,E(W)}\!y_e\:=\:\sum_{e\,\in\,E(W)}\frac{1}{a_e}\:\le\:
\frac{\beta\,|W|}{3\beta+9\zeta}+\frac{\zeta\,|W|}{\beta+3\zeta}\:=\:
\frac{\beta+3\zeta}{3\beta+9\zeta}\,|W|\:=\:\tfrac13\,|W|.\]
Since $|W|\ge3$, we have $\frac13\,|W|\le\half\,(|W|-1)$, which completes
the proof of the lemma.\eop

\section{Proof of Theorem~\ref{th2}}\label{proof-cl}

We use the notation and terminology from Section~\ref{proof}.

We start similarly to the proof of Theorem~\ref{th1} in
Subsection~\ref{sec2.1}. Suppose Theorem~\ref{th2} is false. Then there
exists a surface~$S$ such that for any $\beta_S,\gamma_S$ we can find
$\beta\ge\beta_S$ and a graph~$G$, with a $\Sigma$-system of width at most
$\beta$, such that $\omega(G;\Sigma)>\tfrac32\, \beta+\gamma_S$. Let
$\zeta_S^*=132\,(3-\chi(S))$ be as given in Lemma~\ref{lem1}. We take
$\zeta_S=\zeta_S^*$, $\beta_S=\frac23\,(\zeta_S^*)^2=11616\,(3-\chi(S))^2$,
and $\gamma_S=\half\,\zeta_S^*+10=208-66\,\chi(S)$. Note that
$\chi(S)\le2$, so $\beta_S\ge11616$.

By assumption, there exist $\beta\ge\beta_S$ and a graph~$G$, with a
$\Sigma$-system of width at most~$\beta$, containing a $\Sigma$-clique
having more than $\tfrac32\,\beta+\gamma_S$ vertices. Choose such graph~$G$
with the minimum number of vertices, and, with respect to that, with the
maximum number of edges.

Similarly as in the proof of Theorem~\ref{th1}, we can assume~$G$ is
connected, has at least 17424 vertices, and is edge-maximal with respect to
being embeddable in~$S$. By Lemma~\ref{lem-1} we get that each vertex has
degree at least three.

The following is an easy observation.

\begin{claim}\label{cl5}
  For any vertex~$v$, every $\Sigma$-clique containing~$v$ has size at most
  $1+\ds(v)$.
\end{claim}

\noindent
Next we prove the following.

\begin{claim}\label{cl6}
  Let adjacent vertices~$v,u$ satisfy $d(v)\le5$ and $d(u)\le\zeta_S$.
  Then~$v$ is in every $\Sigma$-clique of size larger than
  $\frac32\,\beta+\gamma_S$, and $\ds(v)\ge\frac32\,\beta+\gamma_S$.
\end{claim}

\begin{proof}
  The argument is similar to the one in Subsection~\ref{ss-S2}: Construct a
  graph~$G_2$ by contracting the edge~$vu$ into a new vertex~$w$ (remove
  multiple edges if they appear). Set $V_2=(V\setminus\{v,u\})\cup\{w\}$.
  Let $\Sigma_2(w)=\Sigma(u)\cup\Sigma(v)\setminus\{u,v\}$. For a vertex
  $t\in V_2\setminus\{w\}$, if $\Sigma(t)$ contains~$u$, then set
  $\Sigma_2(t)=(\Sigma(t)\setminus\{u,v\})\cup\{w\}$; otherwise, set
  $\Sigma_2(t)=\Sigma(t)\setminus\{v\}$. Note that~$G_2$ is smaller
  than~$G$ and is still embeddable in~$S$. Moreover, for every
  $t\in V_2\setminus\{w\}$, we have $|\Sigma_2(t)|\le|\Sigma(t)|\le\beta$;
  while for~$w$ we have $|\Sigma_2(w)|\le|\Sigma(u)|+|\Sigma(v)|\le
  d_G(u)+d_G(v)\le5+\zeta_S\le\beta$.

  By construction, it is easy to check that every $\Sigma$-clique in~$G$
  not containing~$v$ corresponds to a $\Sigma_2$-clique in~$G_2$ of the
  same size. Since~$G$ was chosen as a smallest counterexample, this means
  that every $\Sigma$-clique in~$G$ of size larger than
  $\frac32\,\beta+\gamma_S$ must contain~$v$.

  For the second part we use that~$G$, as a counterexample, must contain
  $\Sigma$-cliques larger than $\frac32\,\beta+\gamma_S$, whereas any
  $\Sigma$-clique in~$G$ containing~$v$ has size at most $1+\ds(v)$.
\end{proof}

\noindent
We continue going through the cases of Lemma~\ref{lem1}. If all vertices
of~$G$ have degree at most~$\zeta_S$, then the number of
$\Sigma$-neighbours of any vertex is at most~$(\zeta_S)^2$. So the maximum
size of a $\Sigma$-clique is at most
$(\zeta_S)^2+1\le\frac32\,\beta+1\le\frac32\,\beta+\gamma_S$, a
contradiction.

Next suppose there is a vertex~$v$ of degree at most five with at most one
neighbour of degree more than~$\zeta_S$. Then, since
$\beta\ge\tfrac23\,(\zeta_S^*)^2\ge8\zeta_S^*$, we have
\[\ds(v)\:\le\:d(v)+\!\!\!
\sum_{t\,\in\,N(v),\;v\,\in\,\Sigma(t)}\!\!\!(|\Sigma(t)|-1)\:\le\:
5+4\,(\zeta_S^*-1)+(\beta-1)\:=\:4\zeta_S^*+\beta\:\le\:\tfrac32\,\beta.\]
But every vertex has degree at least three, hence~$v$ has a neighbour~$u$
of degree at most~$\zeta_S^*$. We obtain a contradiction with
Claim~\ref{cl6}.

\medskip
Let~$X$ and~$Y $ be the two disjoint, non-empty, sets forming a very
special $\zeta_S^*$-pair in~$G$ satisfying~(S3) in Lemma~\ref{lem1}. For
convenience, we repeat the essential properties of those sets:{

  \smallskip
  \qitee{(i)}\emph{Every vertex in~$X$ has degree at least $\zeta_S^*+1$.
    Every vertex $y\in Y$ has degree four, is adjacent to exactly two
    vertices of~$X$, and the remaining neighbours of~$y$ have degree four
    as well.}

  \smallskip
  \qitee{(ii)}\emph{For all pairs of vertices $y,z\in Y$, if~$y$ and~$z$
    are adjacent or have a common neighbour $w\notin X$, then $X^y=X^z$.}

  \smallskip
  \qitee{(iii)}\emph{For all non-empty subsets $W\subseteq X$, we have
    $e(W,V\setminus Y)\le e(W,Y\setminus Y^W)+\zeta_S^*\,|W|$.}

}\smallskip\noindent
We can remove from~$X$ any vertex not adjacent to any vertex in~$Y$.

We can use arguments similar to the first part of Subsection~\ref{ss-S3} to
show the following.

\begin{claim}\label{cl2a}
  For all $y\in Y$, we have that if $X^y=\{x_1,x_2\}$, then
  $y\in\Sigma(x_1)\cap\Sigma(x_2)$.
\end{claim}

\noindent
Next, by~(i), every $y\in Y$ has degree four and a neighbour~$u$ of degree
four. From Claim~\ref{cl6} we can conclude:

\begin{claim}\label{cl7}
  For every $y\in Y$, we have that~$y$ is in every $\Sigma$-clique of size
  larger than $\frac32\,\beta+\gamma_S$, and
  $\ds(y)\ge\frac32\,\beta+\gamma_S$.
\end{claim}

\noindent
Also by the properties of the vertices in~$Y$ according to~(i) and~(ii), we
have for all $y\in Y$ and $X^y=\{x_1,x_2\}$,
\begin{align*}
  \ds(y)\:&\le\:4+2\cdot(4-1)+|\Sigma(x_1)\setminus\{y\}|+
  |\Sigma(x_2)\setminus\{y\}|-|Y^{\{x_1,x_2\}}\setminus\{y\}|\\
  &=\:9+|\Sigma(x_1)|+|\Sigma(x_2)|-|Y^{\{x_1,x_2\}}|
\end{align*}
Here we use that by Claim~\ref{cl2a} all vertices in $Y^{\{x_1,x_2\}}$ are
contained in both $\Sigma(x_1)$ and $\Sigma(x_2)$; hence we can subtract
the term $|Y^{\{x_1,x_2\}}\setminus\{y\}|$, since these vertices are
counted twice in $|\Sigma(x_1)\setminus\{y\}|+|\Sigma(x_2)\setminus\{y\}|$.
Since $|\Sigma(x_1)|,|\Sigma(x_2)|\le\beta$, from Claim~\ref{cl7} we can
conclude the following.

\begin{claim}\label{cl8}
  For every pair $x_1,x_2\in X$ for which there is a $y\in Y$ with
  $X^y=\{x_1,x_2\}$, we have $|Y^{\{x_1,x_2\}}|\le\half\,\beta-\gamma_S+9$.
\end{claim}

\noindent
Since every vertex in~$Y$ is in every $\Sigma$-clique of size larger than
$\frac32\,\beta+\gamma_S$, and by the hypothesis there is at least one such
clique, we must have that all pairs of vertices in~$Y$ are adjacent or
appear together in some $\Sigma(v)$. By~(ii), this proves that for every
two vertices $y_1,y_2\in Y$, we have $X^{y_1}\cap X^{y_2}\ne\varnothing$.
As a consequence, if~$G_X$ denotes the graph with vertex set~$X$ in which
two vertices are adjacent if they have a common neighbour in~$Y$,
then~$G_X$ is either a triangle or a star. (Here we use that we can assume
all vertices in~$X$ to have at least one neighbour in~$Y$.)

\medskip\noindent
\textbf{Case 1}.\quad$G_X$ is a triangle.\\*
Let $X=\{x_1,x_2,x_3\}$. This means that
$Y=Y^{\{x_1,x_2\}}\cup Y^{\{x_1,x_3\}}\cup Y^{\{x_2,x_3\}}$, and so by
Claim~\ref{cl8} we get $|Y|\le\frac32\,\beta-3\gamma_S+27$.

Since $Y^X=Y$ by definition of~$X$, we have $e(X,Y\setminus Y^X)=0$. So
using the inequality in~(iii) with $W=X$ leads to
$e(X,V\setminus Y)\le3\zeta_S^*$. That means there must be~$x_{j_1}$
and~$x_{j_2}$ such that $e(\{x_{j_1},x_{j_2}\},V\setminus Y)\le2\zeta_S^*$.
And so for $y\in Y^{\{x_{j_1},x_{j_2}\}}$, we can estimate, using~(i) and
$|X|=3$,
\begin{align*}
  \ds(y)\:&\le\:2+2\cdot(4-1)+|X|+(|Y|-1)+
  e(\{x_{j_1},x_{j_2}\},V\setminus(X\cup Y))\\
  &\le\:\tfrac32\,\beta-3\gamma_S+37+2\zeta_S^*.
\end{align*}
But this contradicts Claim~\ref{cl7}, since $4\gamma_S>2\zeta_S^*+37$.

\medskip\noindent
\textbf{Case 2}.\quad$G_X$ is a star.\\*
We denote by~$x$ the vertex of~$X$ corresponding to the centre of the
star~$G_X$, and by $x_1,\ldots,x_k$, $k\ge1$, the vertices of~$X$
corresponding to the leaves.

Using the inequality in~(iii) with $W=X$ again, we get
$e(X,V\setminus Y)\le\zeta_S^*\,|X|=(k+1)\,\zeta_S^*$. Since
$X=\{x,x_1,\ldots,x_k\}$, there must be an~$x_j$ such that
$e(\{x_j\},V\setminus Y)\le\dfrac1k\,(k+1)\,\zeta_S^*\le2\zeta_S^*$. Now
for $y\in Y^{\{x,x_j\}}$, we can estimate
\[\ds(y)\:\le\:4+2\cdot(4-1)+|(\Sigma(x)\cup\Sigma(x_j))\setminus\{y\}|
\:=\:9+|\Sigma(x)|+|\Sigma(x_j)\setminus\Sigma(x)|.\]
Since $Y\subseteq\Sigma(x)$, we have
$|\Sigma(x_j)\setminus\Sigma(x)|\le e(\{x_j\},V\setminus Y)\le 2\zeta_S^*$.
Together with $|\Sigma(x)|\le\beta$, this means
$\ds(y)\le\beta+9+2\zeta_S^*$. This contradicts Claim~\ref{cl7}, since
$\half\,\beta>9+2\zeta_S^*$.\eop

\bigskip\noindent
In the proof of Theorem~\ref{th2}, we used $\beta_S=11616\,(3-\chi(S))^2$
and $\gamma_S=208-66\,\chi(S)$. Since the sphere~$\mathbb{S}^2$ has
$\chi(\mathbb{S}^2)=2$, following the proof above means we can obtain
$\beta_P=11616$ and $\gamma_P=76$ for the planar case. But it is clear that
these values are far from best possible. Using more careful estimates in
the proof above and more careful reasoning in certain parts of the proof of
Lemma~\ref{lem1} can give significantly smaller values. Since our first
goal is to show that we can obtain constant values for these results, we do
not pursue this further.

\section{Concluding Remarks and Discussion}\label{conclusion}

\subsection{About the Proof}

The proof of our main theorem for major parts follows the same lines as the
proof of Theorem~\ref{mt2} in~\cite{HHMR}. In particular, the proof of that
theorem also starts with a structural lemma comparable to Lemma~\ref{lem1},
uses the structure of the graph to reduce the problem to edge-colouring a
specific multigraph, and then applies (and extends) Kahn's approach to that
multigraph. Of course, a difference is that Theorem~\ref{mt2} only deals
with list colouring the square of a graph, but it is probably possible to
generalise the whole proof to the case of list $\Sigma$-colouring.
Nevertheless, there are some important differences in the proofs we feel
deserve highlighting.

Lemma~\ref{lem1} is stronger than the comparable~\cite[Lemma~3.3]{HHMR}. We
obtain a set~$Y$ of vertices with degree four and with a very specific
structure of their neighbourhoods. This structure allows us to construct a
multigraph~$H$ so that a standard list edge-colouring of~$H$ provides the
information to colour the vertices in~$Y$ (see Lemma~\ref{lem:ext}). In the
lemma in~\cite{HHMR}, the vertices in the comparable set~$Y$ are only
guaranteed to have degree at most~$\Delta^{1/4}$, and knowledge about their
neighbourhood is far sketchier. This means that the translation to list
edge-colouring of a multigraph is not so clean; apart from the normal
condition in the list edge-colouring of~$H$ (that adjacent edges need
different colours), for each edge there may be up to~$O(\Delta^{1/2})$
non-adjacent edges that also need to get a different colour. In particular
this means that in~\cite{HHMR}, Kahn's result in Theorem~\ref{kahn-main}
cannot be used directly. Instead, a new, stronger, version has to be proved
that can deal with a certain number of non-adjacent edges that need to be
coloured differently. Lemma~\ref{lem1} allows us to use Kahn's Theorem
directly.

A second aspect in which our Lemma~\ref{lem1} is stronger is that in the
final condition~(S3), we have an `error term' that is a constant
times~$|W|$. In~\cite{HHMR} the comparable term is $\Delta^{9/10}\,|W|$,
where~$\Delta$ is the maximum degree of the graph. This in itself already
means that the approach in~\cite{HHMR} at best can give a bound of the type
$\frac32\,\Delta+o(\Delta)$. The fact that we cannot do better with the
stronger structural result is because of the limitations of Kahn's Theorem,
Theorem~\ref{kahn-main}. If it would be possible to replace the condition
in that theorem by a condition of the form `the vector $\vec{x}=(x_e)$ with
$x_e=\tfrac1{|L(e)|-K}$ for all $e\in E(H)$ is an element of~$\MP(H)$',
where~$K$ is some positive constant, the work in this paper would directly
give an improvement for the bound in Theorem~\ref{th1} to
$\frac32\,\beta+O(1)$. Note that our version of Lemma~\ref{lem-mp} is also
already strong enough to support that case.

Lemma~\ref{lem1} also allows us to prove a bound $\frac32\,\beta+O(1)$ for
the $\Sigma$-clique number in Theorem~\ref{th2}. The important corollary
that the square of a graph embeddable in a fixed surface has clique number
at most $\frac32\,\Delta+O(1)$ would have been impossible without the
improved bound in the lemma.

Also Lemma~\ref{lem-mp} is stronger than its
compatriot~\cite[Lemma~5.9]{HHMR}. The lemma in~\cite{HHMR} only deals with
the case $d_G(v)=\beta$ for all vertices~$v$ in~$H$. Because of this, it
can only be applied to the case that all vertices in~$H$ have maximum
degree~$\Delta(G)$ in~$G$. Some non-trivial trickery then has to be used to
deal with the case that there are vertices in~$H$ of degree less
than~$\Delta(G)$ in~$G$. Moreover, the proof of Lemma~\ref{lem-mp} is
completely different from the proof in~\cite{HHMR}. We feel that our new
proof is more natural and intuitive, giving a clear relation between the
lower bounds on the sizes of the lists and the upper bound of the sum of
their inverses. The proof in~\cite{HHMR} is more ad-hoc, using some
non-obvious distinction in a number of different cases, depending on the
size of~$W$ and the degrees of some vertices in~$W$.

\subsection{Further Work}

We feel that our work is just the beginning of the study of general
$\Sigma$-colouring problems. It should be possible to obtain deeper results
taking into account the structure of the $\Sigma$-system, and not just the
sizes of the sets~$\Sigma(v)$. The following easy result is an example of
this.

Recall that a graph is \emph{$q$-degenerate} if there exists an ordering
$v_1,v_2,\dots,v_n$ of the vertices such that every~$v_i$ has at most~$q$
neighbours in $\{v_1,\dots,v_{i-1}\}$. A class of graphs is degenerate if
there is some~$q$ such that every graph in the class is $q$-degenerate.
Examples of degenerate graph classes are graphs embeddable in a fixed
surface, and proper minor-closed classes.

\begin{proposition}\label{thm3}\mbox{}\\*
  For any degenerate graph class~$\MF$, there exists a constant~$c_\MF$
  such that the following holds. Let~$G$ be a graph in~$\MF$, together with
  a $\Sigma$-system so that $\Sigma(u)\cap\Sigma(v)=\varnothing$ for every
  two distinct vertices~$u,v$. Then
  $\ch(G;\Sigma)\le\Delta(G;\Sigma)+c_\MF$.
\end{proposition}

\begin{trivlist}\item[]{\bf Proof}\mbox{ \ }Suppose every graph
  in~$\MF$ is $q$-degenerate, and set $c_\MF=q+1$. For a graph~$G$
  in~$\MF$, take an ordering $v_1,\ldots,v_n$ of its vertices such that
  each~$v_i$ has at most~$q$ neighbours in $\{v_1,\ldots,v_{i-1}\}$. We
  greedily colour the vertices $v_1,\ldots,v_n$ in~$G$ in that order.

  Note that by the hypothesis, each vertex~$v$ has at most one
  neighbour~$w$ with $v\in\Sigma(w)$. When colouring the vertex~$v_i$, we
  need to take into account its neighbours in $\{v_1,\ldots,v_{i-1}\}$,
  plus the vertices in $\Sigma(w)\cap\{v_1,\ldots,v_{i-1}\}$ for a
  vertex~$w$ with $v_i\in\Sigma(w)$ (where that vertex~$w$ can be in
  $\{v_{i+1},\ldots,v_n\}$). By construction of the ordering, there are at
  most~$q$ neighbours of~$v_i$ in $\{v_1,\ldots,v_{i-1}\}$. And a
  vertex~$w$ with $v_i\in\Sigma(w)$ has at most
  $|\Sigma(w)|\le\Delta(G;\Sigma)$ vertices in
  $\Sigma(w)\cap\{v_1,\ldots,v_{i-1}\}$. So the total number of forbidden
  colours when colouring~$v_i$ is at most $\Delta(G;\Sigma)+q$. Since each
  vertex has $\Delta(G;\Sigma)+q+1$ colours available, the greedy algorithm
  will always find a free colour.\eop
\end{trivlist}

\noindent
We think that it is possible to combine our main theorem and the theorem
above in the following way. For a $\Sigma$-system for a graph~$G$, let
$k(G;\Sigma)$ be the maximum of $|\Sigma(u)\cap\Sigma(v)|$ over all
pairs~$u,v$ of distinct vertices.

\begin{conjecture}\label{con4.2}\mbox{}\\*
  Let~$S$ be a fixed surface. Then there exists a constant~$c_S$ such that
  for all graphs~$G$ embeddable in~$S$, with a $\Sigma$-system, we have
  \[\ch(G;\Sigma)\:\le\:\Delta(G;\Sigma)+k(G;\Sigma)+c_S.\]
\end{conjecture}

\noindent
This conjecture would fit with our current proof of Theorem~\ref{th1}, the
main part of which is a reduction of the original problem to a list
edge-colouring problem. For this approach, Shannon's Theorem~\cite{Sh49}
that a multigraph with maximum degree~$\Delta$ has an edge-colouring using
at most $\bigl\lfloor\frac32\,\Delta(G)\bigr\rfloor$ colours, forms a
natural base for the bounds conjectured in Conjecture~\ref{con1}. If the
relation between colouring the square of graphs embeddable in a fixed
surface and edge-colouring multigraphs holds in a stronger sense, then
Conjecture~\ref{con4.2} forms a logical extension of Vizing's
Theorem~\cite{Viz64} that a multigraph with maximum degree~$\Delta$ and
maximum edge-multiplicity~$\mu$ has an edge-colouring with at most
$\Delta+\mu$ colours.

In Borodin \emph{et al.}~\cite{Bor07}, a weaker version of
Conjecture~\ref{con4.2} for cyclic colouring of plane graphs was proved.
Recall that if~$G^P$ is a plane graph, then~$\Delta^*$ is the maximum
number of vertices in a face. Let~$k^*$ denote the maximum number of
vertices that two faces of~$G^P$ have in common.

\begin{theorem}[Borodin, Broersma, Glebov \& Van den
  Heuvel~\cite{Bor07}]\label{th-c-k}\mbox{}\\*
  For a plane graph~$G^P$ with $\Delta^*\ge4$ and $k^*\ge4$, we have
  $\chi^*(G^P)\le\Delta^*+3k^*+2$.
\end{theorem}

\subsubsection*{\boldmath$\Sigma$-Colouring and Minor-Closed Classes.}

It seems natural to expect that our work on graphs embeddable in a fixed
surface can be extended to arbitrary proper minor-closed classes of graphs.
Compare our main Theorem~\ref{th1} with Theorem~\ref{mt2}, the main result
from~\cite{HHMR}. But there exist some obstacles to a direct
generalisation.

It is easy to show that if a graph $G$ is \emph{$q$-degenerate}, then its
square is $((2q-1)\,\Delta(G))$-degenerate. It is well-known, see
e.g.~\cite{Mad68}, that for every proper minor-closed family~$\MF$, there
is a constant~$C_\MF$ such that every graph in~$\MF$ is $C_\MF$-degenerate.
Hence~$G^2$ is $((2C_\MF-1)\,\Delta(G))$-degenerate, and so for every
$G\in\MF$, we have $\ch(G^2)\le(2C_\MF-1)\,\Delta(G)+1$.

For $\Sigma$-colouring, there is no comparable upper bound on
$\ch(G;\Sigma)$ in terms of the degeneracy of~$G$ and $\Delta(G;\Sigma)$.
To see this, let~$G$ be the graph obtained from the complete graph~$K_n$,
$n\ge4$, by subdividing all edges of~$K_n$ once. For a vertex~$v$
corresponding to an original vertex in~$K_n$, set $\Sigma(v)=\varnothing$;
while for a ``new'' vertex~$v$ of degree two, set $\Sigma(v)=N_G(v)$. Then
we have that~$G$ is 2-degenerate and $\Delta(G;\Sigma)=2$, but
$\ch(G;\Sigma)=n$.

Nevertheless, combining the Robertson and Seymour graph minor structure
theorem~\cite{RS03} with our main theorem on graphs embeddable in bounded
genus surfaces, one can fairly easily obtain the following.

\begin{theorem}\label{thm-mcf}\mbox{}\\*
  Let~$\MF$ be a proper minor-closed family of graphs. Then there exist
  constants~$C_\MF$ and~$c_\MF$ such that the following holds: For any
  graph~$G$ in~$\MF$ with a $\Sigma$-system, we have
  $\ch(G;\Sigma)\le C_\MF\,\Delta(G;\Sigma)+c_\MF$.
\end{theorem}

\noindent
Giving more details of the ideas of the proof would require a number of
additional definitions, and is beyond the scope of this short discussion.
It would be interesting to find a proof of this theorem that does not
require the full force of the graph minor structure theorem.

Also finding the smallest possible constant~$C_\MF$ for certain
minor-closed families~$\MF$ appears an interesting question.
Theorem~\ref{mt2} clearly suggests that if~$\MF_k$ denotes the class of
$K_{3,k}$-minor free graphs ($k\ge3$), $C_{\MF_k}$ should be equal to
$3/2$.

\section{Kahn's Work on List Edge-Colourings}\label{background}

As mentioned earlier, Theorem~\ref{kahn-main} is not explicitly stated
in~\cite{Kah00}, but is implicit in the proof of the main result of that
paper. In this final section, we give an overview of how this theorem can
be obtained from the ideas in Kahn's paper.

The main result in~\cite{Kah00} is that the list chromatic index is
asymptotically equal to the fractional chromatic index of a multigraph.

\begin{theorem}[Kahn~\cite{Kah00}]\label{kahn-orig}\mbox{}\\*
  For any $\eps>0$, there exists a~$\Delta_\eps$ such that for all
  $\Delta\ge\Delta_\eps$ the following holds. If~$H$ is a multigraph with
  maximum degree at most~$\Delta$, then
  \[\chi'_f(H)\:\le\:\chi'(H)\:\le\:\ch'(H)\:\le\:(1+\eps)\,\chi'_f(H).\]
\end{theorem}

\bigskip\noindent
Here~$\chi'(H)$ is the normal chromatic index (or edge-chromatic number)
of~$H$, $\chi'_f(H)$ is the fractional chromatic index of~$H$,
and~$\ch'(H)$ is the list chromatic index of~$H$. The crucial step to
relate this result to the matching polytope~$\MP(H)$ is the following
well-known characterisation of the fractional chromatic index:
\[\chi'_f(H)\:=\:\min\{\,\gamma>0\mid
\text{the vector $(x_e)_{e\in E(H)}$ with $x_e=\gamma^{-1}$ is in
  $\MP(H)$}\,\}.\]
So Theorem~\ref{kahn-orig} is just a special case of
Theorem~\ref{kahn-main} if we set $|L(e)|\ge\dfrac{\chi'_f(H)}{1-\delta}$
for all edges~$e$. (The second condition of Theorem~\ref{kahn-main} is
automatically satisfied in that case, since trivially
$\chi'_f(H)\ge\Delta(H)$.)

In order to prove Theorem~\ref{kahn-orig}, Kahn describes a randomised
iterative procedure that colours the edges of~$H$ in a number of stages.
During this procedure, the lists of available colours for each edge will
change, and the lists will not be the same size for the uncoloured edges.
This is why, roughly speaking, Kahn's actual proof deals with the more
general case, as described in Theorem~\ref{kahn-main}.

In order to give the reader a better understanding of the background of
Kahn's approach, we give an overview of the crucial elements in the
following subsections.

\subsection{Hardcore Distributions}

Hardcore distributions are distributions that originally arose in
Statistical Physics, and that satisfy very natural conditions which
generally provide strong independence properties allowing good sampling
from a given family. Given a family of subsets~$\MF$ of a given
set~$\mathcal{E}$, a natural way of picking at random an element of~$\MF$
(or, in an other words, a probability distribution on~$\MF$) is as follows.

Let us suppose that each element~$e$ of~$\mathcal{E}$ has been assigned a
positive weight~$\lambda_e$. Then we pick each element $M\in\MF$ with
probability proportional to $\prod\limits_{e\,\in\,M}\lambda_e$. More
precisely, the probability~$P_M$ of picking $M\in\MF$ at random is given by
\[P_M\:=\:\frac{\prod\limits_{e\,\in\,M}\lambda_e}
{\sum\limits_{M'\,\in\,\MF}\:\prod\limits_{e\,\in\,M'}\lambda_e}.\]
We define the vector $\vec{x}=(x_e)_{e\in\mathcal{E}}$ by setting
$x_e=\sum\limits_{M\,\in\,\MF,\;e\in M}\!\!P_M$. It is clear that~$x_e$ is
the probability that a given random element of~$\MF$ contains the
element~$e$. The probability distribution $\{P_M\}$ is called a
\emph{hardcore} distribution with \emph{activities~$\{\lambda_e\}$} and
\emph{marginals~$\{x_e\}$}. The vector $\vec{x}$ is called the
\emph{marginal vector} associated with the hardcore distribution~$\{P_M\}$.

Given a vector~$\vec{x}$, it is not always true that~$\vec{x}$ is the
marginal vector of some hardcore distribution. Indeed if $\mathcal{P}(\MF)$
denotes the polytope defined by taking the convex hull of the
characteristic vectors of the elements of~$\MF$\,\footnote{\,Recall that
  the characteristic vector, $\textbf{1}_M$, of a given element $M\in\MF$
  is the $|\mathcal{E}|$-dimensional vector $(y_e)_{e\in\mathcal{E}}$ such
  that $y_e=1$ if $e\in M$ and $y_e=0$ otherwise.}, then the marginal
vector~$\vec{x}$ of a hardcore distribution is in $\mathcal{P}(\MF)$:
\[\vec{x}\:=\:\sum_{M\,\in\,\MF}P_M\,\textbf{1}_{M}.\]

This provides a necessary condition for a vector to be the marginal vector
of a hardcore distribution. It is not difficult to prove that the
activities~$\lambda_e$ corresponding to~$\vec{x}$, if they exist, are
unique.

\medskip
{}From now on, let~$H$ be a given multigraph. We recall that~$\MM(H)$
and~$\MP(H)$ are the family of matchings and the matching polytope of~$H$,
respectively. (So~$\MM(H)$ will play the role of the family~$\MF$ from
above. And using the notation from above means
$\MP(H)=\mathcal{P}(\MM(H))$.)

We have the following theorem relating the matching polytope and hardcore
distributions.

\begin{theorem}[Lee~\cite{Lee90}, Rabinovich \emph{et
    al.}~\cite{RSW92}]\label{lem-hardcore}\mbox{}\\*
  For a given real number $0<\delta<1$, suppose~$\vec{x}$ is a vector in
  $(1-\delta)\,\MP(H)$, for some multigraph~$H$. Then there exists a unique
  family of activities~$\lambda_e$ such that~$\vec{x}$ is the marginal
  vector of the hardcore distribution on the matchings of~$H$ defined by
  the~$\lambda$'s. The hardcore distribution $\{P_M\}_{M\in\MM(H)}$ is the
  unique distribution maximising the entropy function
  \[\mathcal{H}({Q_M})\:=\:-\!\sum_{M\,\in\,\MM(H)}\!Q_M\,\log(Q_M)\]
  among all the distributions~$\{Q_M\}_{M\in\MM(H)}$ satisfying
  $\vec{x}=\!\sum\limits_{M\,\in\,\MM(H)}\!Q_M\,\textbf{1}_M$.
\end{theorem}

\noindent
Kahn and Kayll proved in~\cite{KaKa97} a family of results, resulting in a
long-range independence property for the hardcore distributions defined by
a marginal vector~$\vec{x}$ inside $(1-\delta)\,\MP(H)$, see~\cite{Kah00}.
We refer to the original papers of Kahn~\cite{Kah96,Kah00} and Kahn and
Kayll~\cite{KaKa97}, and the book by Molloy and Reed~\cite{MoRe} for more
on these issues. We settle here for citing the following lemma.

\begin{lemma}[{\cite[Lemma 4.1]{KaKa97}}]\label{lem-hd1}\mbox{}\\*
  For every~$\delta$, $0<\delta<1$, there is a $\rho_\delta>0$ such that if
  $\{P_M\}$ is a hardcore distribution on the matchings of~$H$ with
  marginal vector $\vec{x}\in(1-\delta)\,\MP(H)$, then for all $u,v\in
  V(H)$,
  \[\mathbf{Pr}(\text{$M$ does not touch $u$ and $v$})\:>\:\rho_\delta.\]
\end{lemma}

\subsection{Hardcore Distributions and Edge-Colouring}\label{sec6.2}

We present here Kahn's algorithm for list edge-colouring of multigraphs
first introduced and analysed in~\cite{Kah00}. We continue to use the
notation of the previous subsection. In particular, we suppose that~$H$ is
a multigraph and~$L$ a list assignment of colours to the edges of~$H$ so
that the conditions of Theorem~\ref{kahn-main} are satisfied. By
Lemma~\ref{lem-hardcore} there exists a hardcore distribution~$\{P_M\}$
with marginals $\{|L(e)|^{-1}\}_{e\in E(H)}$, which in addition satisfies
the property of Lemma~\ref{lem-hd1}. Let~$\{\lambda_e\}$ be the activities
on the edges (which are unique by Theorem~\ref{lem-hardcore}) corresponding
to this distribution. An extra property is indeed true: For every
subgraph~$H^*$ of~$H$ it is possible to find a hardcore distribution
$\{P^*_M\}$ with corresponding marginals $|L(e)|^{-1}$ for $e\in E(H^*)$.
The corresponding activities~$\lambda^*_e$ will in general be different
from the~$\lambda_e$'s.

The algorithm works as follows: Let
$\mathcal{L}=\!\bigcup\limits_{e\in E(H)}\!L(e)$ be the union of the
colours in the lists. For each colour~$\alpha$, let us define the
\emph{colour graph~$H_\alpha$} to be the graph containing all the edges
whose lists contain the colour~$\alpha$. And denote
by~$\{\lambda_{\alpha,e}\}$ the activities producing the hardcore
distribution with marginals $|L(e)|^{-1}$ for $e\in E(H_\alpha)$. The
colouring procedure consists in a finite number of iterations of a
procedure that we may call \emph{naive colouring}. At step~$i$ of the
iteration, we are left with subgraphs~$H^i_\alpha$ containing some
uncoloured edges whose lists contain the colour~$\alpha$. Of course we have
$H^i_\alpha\subseteq H^{i-1}_\alpha\subseteq\dots\subseteq
H^0_\alpha=H_\alpha.$

The naive colouring procedure at step $i+1$ consists of the following
sub-steps.{

  \smallskip
  \qitee{(a)}For each colour $\alpha\in\mathcal{L}$, choose independently
  of the other colours a random matching
  $M^{i+1}_\alpha\subseteq E(H^i_\alpha)$ according to the hardcore
  distribution defined by the activities~$\lambda_{\alpha,e}$ on the edges
  $e\in E(H^i_\alpha)$.

  \smallskip
  \qitee{(b)}If an edge~$e$ is in one or more of the matchings
  $M^{i+1}_\alpha$ chosen above, then choose one of the colours from those,
  chosen uniformly at random, and colour~$e$ with that colour.

  \smallskip
  \qitee{(c)}For each colour~$\alpha$, form $H^{i+1}_\alpha$ by removing
  from~$H^i_\alpha$ all the edges that received some colour at this stage,
  and all vertices that are incident to one of the edges coloured
  with~$\alpha$. (While removing a vertex, all the edges incident to it are
  of course removed as well.)

}\smallskip\noindent
Note that the process above can be described also in terms of
subgraphs~$H^i$ of the original multigraph~$H$, where the edges of~$H^i$
are the edges that are still uncoloured after step~$i$, and each edge~$e$
in~$H^i$ has a list of colours~$L^i(e)$ formed by all colours~$\alpha$ for
which $e\in E(H^i_\alpha)$. Also note that the
activities~$\lambda_{\alpha,e}$ remain unchanged all through the process.

A sufficient number of iterations of the naive colouring procedure results
in a graph~$H^I$, consisting of all the uncoloured edges at this step, such
that~$H^I$ has maximum degree~$T$, for some integer~$T$, and that the list
sizes are at least~$2T$, i.e., each uncoloured edge is in at least~$2T$ of
the $H^I_\alpha$'s. (Remember that the conditions of
Theorem~\ref{kahn-main} imply that the lists are quite large at the
beginning.) At this stage it is easy to finish the procedure by a simple
greedy algorithm.

\medskip
The heart of the analysis of the above algorithm in Kahn's approach is the
following strong lemma, the proof of which can be found in~\cite{Kah00}.
(To avoid confusion between an edge~`$e$' and the base of the natural
logarithms~2.718\ldots, we will use a roman~`$\mathrm{e}$' for the latter
one.)

\begin{lemma}[Kahn {\cite[Lemma~3.1]{Kah00}}\,]\label{keylem}\mbox{}\\*
  For each $K>0$ and $0<\eta<1$, there are constants
  $0<\xi_{K,\eta}\le\eta$ and $\Delta_{K,\eta}$ such that the following
  holds for all $\Delta\ge\Delta_{K,\eta}$. Let~$H$ be a multigraph with
  lists~$L(e)$ of colours for each edge~$e$. For each colour~$\alpha$,
  define the colour graph~$H_\alpha$ as above. Finally, for each
  colour~$\alpha$ we are given a hardcore distribution on the matchings
  of~$H_\alpha$ with activities $\{\lambda_{\alpha,e}\}_{e\in E(H_\alpha)}$
  and marginals $\{x_{\alpha,e}\}_{e\in E(H_\alpha)}$. Suppose the
  following conditions are satisfied:{

    \smallskip
    \qite{$\bullet$}for every vertex~$v$, $d_H(v)\le\Delta$;

    \smallskip
    \qite{$\bullet$}for every colour~$\alpha$ and edge $e\in E(H_\alpha)$,
    $\lambda_{\alpha,e}\le\dfrac{K}\Delta$; and

    \smallskip
    \qite{$\bullet$}for every edge~$e$,
    $1-\xi_{K,\eta}\le\sum\limits_{\alpha\,in\,L(e)}\!x_{\alpha,e}\le
    1+\xi_{K,\eta}$.

  }\smallskip\noindent
  Then with positive probability the naive colouring procedure described
  above gives matchings $M_\alpha\subseteq E(H_\alpha)$ for all
  colours~$\alpha$, so that if we set $H^*=H-\bigcup_{\alpha'}M_{\alpha'}$,
  $H^*_\alpha=H_\alpha-V(M_\alpha)-\bigcup_{\alpha'}M_{\alpha'}$, and form
  lists~$L^*(e)$ for all edges $e\in E(H^*)$ by removing no longer allowed
  colours from~$L(e)$, we have:{

    \smallskip
    \qite{$\bullet$}for every vertex~$v$,
    $d_{H^*}(v)\le
    \dfrac{1+\eta}{1+\xi_{K,\eta}}\,\mathrm{e}^{-1}\,\Delta$; and

    \smallskip
    \qite{$\bullet$}for every edge~$e$ in~$H^*$,
    $1-\eta\le\sum\limits_{\alpha\,\in\,L^*_{\strut}(e)}\!
    x^*_{\alpha,e}\le 1+\eta$.

  }\noindent
  Here $\{x^*_{\alpha,e}\}_{e\in E(H^*_\alpha)}$ are the marginals
  associated to $\lambda_{\alpha,e}$ in~$H^*_\alpha$.
\end{lemma}

\noindent
In other words, the lemma guarantees that after one iteration of the naive
colouring procedure, with positive probability the multigraph formed by the
uncoloured edges has maximum degrees bounded by
$\dfrac{1+\eta}{1+\xi_{K,\eta}}\,\mathrm{e}^{-1}\,\Delta$, while the sum of
the marginal probabilities~$x^*_{\alpha,e}$ for every edge~$e$ will be
close to~1.

In the next subsection we will combine all the strands and use the lemma
above to conclude the proof of Theorem~\ref{kahn-main}.

\subsection{Completing the Proof of Theorem~\ref{kahn-main} --- after Kahn}

Let $0<\delta<1$ and $\nu>0$. We must prove the existence of
a~$\Delta_{\delta,\nu}$ such that for $\Delta\ge\Delta_{\delta,\nu}$ the
following holds. Let~$H$ be a multigraph and~$L$ a list assignment of
colours to the edges of~$H$ so that{

  \smallskip
  \qite{$\bullet$}for every vertex~$v$, $d_H(v)\le\Delta$;

  \smallskip
  \qite{$\bullet$}for all edges $e\in E(H$), $|L(e)|\ge\nu\,\Delta$;

  \smallskip
  \qite{$\bullet$}the vector $\vec{x}=(x_e)$ with $x_e=\dfrac1{|L(e)|}$ for
  all $e\in E(H)$ is an element of $(1-\delta)\,\MP(H)$.

}\smallskip\noindent
Then there should exist a proper edge-colouring of~$H$, where each edge
receives a colour from its own list.

For each colour~$\alpha$, define the colour graph~$H_\alpha$ as in the
previous subsection. For each colour~$\alpha$ and edge~$e$, set
$x_{\alpha,e}=x_e=\dfrac1{|L(e)|}$, and let
$\{\lambda_{\alpha,e}\}_{e\in E(H_\alpha)}$ be the activities associated
with the marginals~$x_{\alpha,e}$ on~$H_\alpha$.

Since for every edge~$e$ we have
$\sum\limits_{\alpha\,\in\,L(e)}\!x_{\alpha,e}=
\sum\limits_{\alpha\,\in\,L(e)}\!|L(e)|^{-1}=1$, we certainly know that{

  \smallskip
  \qite{$\bullet$}for every edge~$e$ and $\xi>0$,
  $1-\xi\le\sum\limits_{\alpha\,\in\,L(e)}\!x_{\alpha,e}\le1+\xi$.

}\smallskip\noindent
We next bound the activities~$\lambda_{\alpha,e}$, using
Lemma~\ref{lem-hd1}. First observe that for all~$\alpha$ the vector
$(x_{\alpha,e})_{e\in E(H_\alpha)}$ is in $(1-\delta)\,\MP(H_\alpha)$. So
by Lemma~\ref{lem-hd1} there is a constant~$\rho_\delta$ such that
if~$M_\alpha$ is chosen according to the hardcore distribution with
marginals $\{x_{\alpha,e}\}$ on~$H_\alpha$, then for all
$u,v\in E(H_\alpha)$, we have $\mathbf{Pr}(\text{$M_\alpha$ does not touch
  $u$ and $v$})>\rho_\delta$. Let $e=uv$ be an edge of~$H_\alpha$. We have
\[x_{\alpha,e}\:=\:\mathbf{Pr}(\text{$M_\alpha$ contains $e$})\:=\:
\lambda_{\alpha,e}\,\mathbf{Pr}(\text{$M_\alpha$ does not touch $u$ and
  $v$})\:>\: \lambda_{\alpha,e}\,\rho_\delta.\]
Given the fact that $x_{\alpha,e}=\dfrac1{|L(e)|}$ and
$|L(e)|\ge\nu\,\Delta$, and setting
$K_{\delta,\nu}=\dfrac1{\rho_\delta\,\nu}$, we infer that
$\lambda_{\alpha,e}<\dfrac{x_{\alpha,e}}{\rho_\delta}\le
\dfrac1{\rho_\delta\,\nu\,\Delta}=\dfrac{K_{\delta,\nu}}\Delta$.

We have shown that there exists a $K_{\delta,\nu}>0$ such that{

  \smallskip
  \qite{$\bullet$}for every colour~$\alpha$ and edge $e\in E(H_\alpha)$,
  $\lambda_{\alpha,e}\le\dfrac{K_{\delta,\nu}}\Delta$.

}\smallskip\noindent
Suppose we repeat the naive colouring procedure from the previous
subsection $s=s_{K_{\delta,\nu}}$ times (where~$s_{K_{\delta,\nu}}$ is a
fixed constant to be made more precise later). Let~$H^i$ be the subgraph
of~$H$ formed by the edges that are as yet uncoloured at step~$i$, and for
each $e\in E(H^i)$ let~$L^i(e)$ be the list of colours from~$L(e)$ that are
still allowed for~$e$ at that stage.

Let $\eta_s=1-\mathrm{e}^{-1}$, and recursively in the up-to-down order for
$i=s-1,\dots,1$, set $\eta_i=\xi_{K_{\delta,\nu},\eta_{i+1}}$, where
$\xi_{K_{\delta,\nu},\eta_{i+1}}$ is the function given by
Lemma~\ref{keylem}. Let $\Delta_{\delta,\nu}=
\max\limits_{i=1,\dots,s}\Delta_{K_{\delta,\nu},\eta_i}$
($\Delta_{K_{\delta,\nu},\eta_i}$ according to Lemma~\ref{keylem} again),
and $\eta_0=0$. By applying Lemma~\ref{keylem} and the observations above,
we can ensure inductively, starting from $i=0$, that for
$\Delta\ge\Delta_{\delta,\nu}$, with positive probability the following
conditions are satisfied for all $i=0,\ldots,s$.{

  \smallskip
  \qite{$\bullet$}For all vertices~$v$, $d_{H^{i}}(v)\le T_i$, where
  $T_0=\Delta$ and
  $T_i=\dfrac{1+\eta_i}{1+\eta_{i-1}}\,\mathrm{e}^{-1}\,T_{i-1}$ for
  $i\ge1$; and

  \smallskip
  \qite{$\bullet$}For all edges $e\in E(H^i)$,
  $1-\eta_i\le\sum\limits_{\alpha\,\in\,L^i(e)}x^i_{\alpha,e}\le1+\eta_i$,
  where $\{x^i_{\alpha,e}\}_{e\in E(H^i)}$ are the marginals associated to
  the hardcore distribution with activities~$\lambda_{\alpha,e}$
  in~$H^i_\alpha$.

}\medskip\noindent
It follows that after~$s$ steps, with positive probability we have{

  \smallskip
  \qite{$\bullet$}for all vertices~$v$,
  $d_{H^s}(v)\le (2-\mathrm{e}^{-1})\,\mathrm{e}^{-s}\,\Delta$; and

  \smallskip
  \qite{$\bullet$}for all edges $e\in E(H^s)$,
  $\mathrm{e}^{-1}\le\sum\limits_{\alpha\,\in\,L^s(e)}\!x^{s}_{\alpha,e}\le
  2-\mathrm{e}^{-1}$.

}\smallskip\noindent
We note that for an edge $e=uv$,
\[x^s_{\alpha,e}\:=\:\lambda_{\alpha,e}\cdot \mathbf{Pr}(\text{$M_\alpha^s$
  does not touch $u$ and $v$})\:\le\:\lambda_{\alpha,e},\]
which implies that
$x^s_{\alpha,e}\le\lambda_{\alpha,e}\le\dfrac{K_{\delta,\nu}}\Delta$. We
infer that for all $e\in E(H^s)$,
\[|L^s(e)|\:=\:|\{\,\alpha\mid e\in E(H^s_\alpha)\,\}|\:\ge\:
\frac\Delta{\mathrm{e}\,K_{\delta,\nu}}.\]
Let $T=\dfrac\Delta{2\mathrm{e}\,K_{\delta,\nu}}$. It is now clear that if
we choose the value of~$s$ in such a way that
$2\mathrm{e}^{-s}\le\dfrac1{2\mathrm{e}\,K_{\delta,\nu}}$ (in other words,
by setting $s=s_{K_{\delta,\nu}}\ge\ln(4K_{\delta,\nu})+1$), we can ensure
with positive probability that
\[d_{H^s}(v)\:\le\:T\quad\text{for all $v\in V(H^s)$}\qquad\text{and}
\qquad|L^s(e)|\:\ge\:2T\quad\text{for each $e\in E(H^s)$}.\]
This finally shows that we can proceed using the greedy algorithm in~$H^s$,
in order to extend the resulting colouring from the naive colouring
procedure in Subsection~\ref{sec6.2} to a colouring of the whole graph.


\end{document}